\pdfoutput=1
\documentclass[11pt,reqno]{amsart}

\usepackage[margin=1in]{geometry}
\IfFileExists{glyphtounicode.tex}{\input{glyphtounicode}\pdfgentounicode=1}{}
\usepackage{cmap}
\usepackage[T1]{fontenc}
\usepackage[utf8]{inputenc}
\usepackage{lmodern}
\usepackage{microtype}
\usepackage{amsmath,amssymb,amsthm,mathrsfs}
\usepackage{enumitem}
\usepackage{booktabs}
\usepackage{cite}
\usepackage[bookmarks=true,bookmarksnumbered=true]{hyperref}
\hypersetup{
  colorlinks=true,
  linkcolor=blue,
  citecolor=blue,
  urlcolor=blue,
  pdftitle={Random Tensor Estimates and Deterministic Diagonal Resolutions for Mixed Paracontrolled Operators},
  pdfauthor={Guangqian Zhao},
  pdfsubject={Centered second-chaos estimates for mixed paracontrolled operators},
  pdfkeywords={Gaussian chaos, random tensors, paracontrolled operators, diagonal contractions, stochastic dispersive equations}
}
\newcommand{\doi}[1]{\href{https://doi.org/#1}{\nolinkurl{doi:#1}}}
\newcommand{\arxiv}[1]{\href{https://arxiv.org/abs/#1}{\nolinkurl{arXiv:#1}}}

\numberwithin{equation}{section}
\theoremstyle{plain}
\newtheorem{theorem}{Theorem}[section]
\newtheorem{proposition}[theorem]{Proposition}
\newtheorem{lemma}[theorem]{Lemma}
\newtheorem{corollary}[theorem]{Corollary}

\theoremstyle{definition}
\newtheorem{definition}[theorem]{Definition}
\newtheorem{assumption}[theorem]{Assumption}
\newtheorem{remark}[theorem]{Remark}

\newcommand{\T}{\mathbb T}
\newcommand{\R}{\mathbb R}
\newcommand{\Z}{\mathbb Z}
\newcommand{\N}{\mathbb N}
\newcommand{\E}{\mathbb E}

\newcommand{\cC}{\mathcal C}
\newcommand{\cI}{\mathcal I}
\newcommand{\cL}{\mathcal L}
\newcommand{\cM}{\mathcal M}
\newcommand{\eps}{\varepsilon}
\newcommand{\la}{\langle}
\newcommand{\ra}{\rangle}
\newcommand{\dd}{\,\mathrm d}
\newcommand{\ii}{\mathrm i}
\newcommand{\one}{\mathbf 1}
\newcommand{\wh}{\widehat}
\newcommand{\op}{\operatorname{op}}
\newcommand{\wick}[1]{:#1:}
\newcommand{\isoW}{\mathbb W}
\newcommand{\Lfin}{L_{\mathrm{fin}}}
\newcommand{\Ret}{\mathsf R}
\newcommand{\Ctr}{\mathsf C}

\title[Random Tensor Estimates]
{Random Tensor Estimates and Deterministic Diagonal Resolutions for Mixed Paracontrolled Operators}
\author{Guangqian Zhao}
\address{School of Mathematical Sciences, University of Science and Technology of China, Hefei, Anhui 230026, China}
\email{zhaoguangqian@mail.ustc.edu.cn}
\subjclass[2020]{60H15, 35L71, 35R60, 60H30, 35Q40}
\keywords{Gaussian chaos; random tensors; paracontrolled operators; diagonal contractions; stochastic dispersive equations}
\date{}

\begin{document}

\begin{abstract}
We prove an operator-level convergence theorem for the mixed paracontrolled
blocks
\[
  T^{i;j,k}_{\Lambda}(w)
  =I_i\bigl(w<\Psi_{j,\Lambda}\bigr)\circ\Psi_{k,\Lambda}
\]
on \(\mathbb T^d\), under Fourier-diagonal Gaussian covariance.  At finite
Galerkin cutoff, the Wick decomposition is
\[
  T^\tau_\Lambda=D^\tau_\Lambda+B^\tau_\Lambda,
  \qquad \tau=(i;j,k),
\]
where \(D^\tau_\Lambda\) is a deterministic Volterra multiplier on the
external diagonal \(n=q\), while \(B^\tau_\Lambda\) is a centered second
homogeneous Gaussian chaos with incidence \(n=q+\ell+r\).  The centered
coefficient has two Gaussian frequency legs and one input-output pair.  Write
\(\Delta_T=\{(t,s):0\le s\le t\le T\}\).  An iterated rectangular non-commutative Khintchine argument, followed by the
four associated oriented flattenings, yields
\[
  \bigl\|B^\tau_{\Lambda,N,Q,M}\bigr\|_
  {L^p\bigl(\Omega;C(\Delta_T;\mathcal L(\ell_q^2,\ell_n^2))\bigr)}
  \lesssim_{p,\varepsilon}
  N^{d/2-\Gamma_\tau+\varepsilon}
  \bigl(M^{d/2}+Q^{d/2}\bigr),
  \qquad
  \Gamma_\tau=\lambda_i+\alpha_j+\alpha_k.
\]
Under the resulting strict Sobolev--Besov summability conditions, the
centered cutoffs converge in \(L^p(\Omega)\) and almost surely in operator
norm.  The covariance contraction is kept separate from the tensor
estimate: retained raw branches, or prescribed finite combinations of
branches, enter through a scalar Volterra criterion imposed after the
relevant cancellation has been formed.  We give a frequency-envelope
variant and verify the diagonal criterion for distinct-speed wave and
Klein--Gordon contractions.
\end{abstract}

\maketitle

\section{Introduction}\label{sec:intro}

Singular stochastic PDEs are commonly formulated through enhanced random
objects which replace products not defined by deterministic distribution
theory.  In parabolic problems this principle underlies the Da
Prato--Debussche method, regularity structures and paracontrolled
distributions \cite{DaPratoDebussche,Hairer,GIP}.  For dispersive equations
some of the enhanced objects are naturally operator-valued: the Duhamel
multiplier retains oscillatory information even when the corresponding
deterministic resonant product is unavailable.  Such operators occur, for
example, in the paracontrolled construction of the three-dimensional
quadratic stochastic wave equation \cite{GKO}.  Random tensor estimates
provide a systematic way to control their Gaussian coefficients
\cite{DNYTensor,Kaneshiro}.

The tensor input is used here in a specific operator geometry.  The two
Gaussian legs, the low input leg and the output leg are kept distinct, and
the incidence relation determines four concrete flattening profiles.  The
probabilistic estimate is then coupled to the lower-chaos Fourier diagonal
and to the Sobolev--Besov summation.  General random tensor inequalities
supply the probabilistic input; the present contribution is the reduction
of the mixed paracontrolled block to this four-leg geometry and its coupling
to a separate Volterra resolution of the deterministic contraction.

We consider the mixed block
\begin{equation}\label{eq:intro-block}
        I_i(w<\Psi_j)\circ\Psi_k .
\end{equation}
Here \(i\) labels the outer Duhamel channel, while \(j\) and \(k\) label the
two stochastic factors.  We use a fixed-aperture low--high paraproduct with
strong frequency separation,
whose aperture is specified in Section~\ref{sec:setup}.  If the Duhamel
multiplier has gain \(\lambda_i\) and \(\Psi_a\) has Fourier amplitude of
order \(\langle n\rangle^{-\alpha_a}\), the usual deterministic resonant
product condition is
\[
        \lambda_i+\alpha_j+\alpha_k>d.
\]
We study the complementary singular range
\begin{equation}\label{eq:det-threshold}
        \lambda_i+\alpha_j+\alpha_k\le d .
\end{equation}

At a finite Galerkin cutoff the Wick identity gives
\begin{equation}\label{eq:intro-wick}
        \Psi_{j,\Lambda}(\ell,s)\Psi_{k,\Lambda}(r,t)
        =\wick{\Psi_{j,\Lambda}(\ell,s)\Psi_{k,\Lambda}(r,t)}
        +\E[\Psi_{j,\Lambda}(\ell,s)\Psi_{k,\Lambda}(r,t)].
\end{equation}
Hence we obtain the exact decomposition
\[
        T^\tau_\Lambda=D^\tau_\Lambda+B^\tau_\Lambda,
        \qquad \tau=(i;j,k).
\]
Fourier-diagonal covariance forces \(\ell+r=0\) in \(D^\tau_\Lambda\),
so that its external variables satisfy \(n=q\).  Thus the lower-chaos term
is a scalar diagonal Volterra multiplier.  The centered term retains the
incidence
\begin{equation}\label{eq:intro-incidence}
        n=q+\ell+r
\end{equation}
and is a second homogeneous Gaussian chaos.

The coefficient tensor of a dyadic centered block has four legs:
\[
        (\ell,\mu),\qquad (r,\lambda),\qquad q,\qquad n,
\]
where the first two are Gaussian legs and the last two are the deterministic
input and output legs.  Proposition~\ref{prop:tensor} controls the random
operator by the four oriented flattenings arising from these indices.  The
incidence \eqref{eq:intro-incidence} then gives, with
\(\Gamma_\tau=\lambda_i+\alpha_j+\alpha_k\),
\begin{equation}\label{eq:intro-profile}
        \|B_{\Lambda,N,Q,M}^{\tau}\|_{\ell_q^2\to\ell_n^2}
        \lesssim
        N^{d/2-\Gamma_\tau}\bigl(M^{d/2}+Q^{d/2}\bigr),
\end{equation}
uniformly in \(\Lambda\), up to logarithmic and arbitrarily small dyadic
losses.  Here \(N\) is the high stochastic scale, \(Q\) the input scale and
\(M\) the output scale.  The proof uses only \(Q\le c_{\rm ap}N\) and
\(M\lesssim N\).  The estimate yields
the admissible window
\begin{equation}\label{eq:intro-window}
\begin{gathered}
        \Gamma_\tau>\frac d2,\qquad
        s<\lambda_i+\Gamma_\tau-d,\\
        \max\{0,d-\Gamma_\tau\}<\sigma
        <\lambda_i+\Gamma_\tau-d,
\end{gathered}
\end{equation}
with strict margins.  Under \eqref{eq:intro-window}, the centered cutoffs
converge in \(L^p(\Omega)\) and almost surely as operators from
\(X_T^\sigma\) to
\(C_TH^{s-\lambda_i}\cap L_T^1B^{\sigma-\lambda_i}_{2,\infty}\).

The diagonal term is treated separately.  We choose a resolution
\(\mathfrak P\), given at cutoff level by
\begin{equation}\label{eq:intro-resolution}
        D^\tau_\Lambda=\Ret^\tau_\Lambda+\Ctr^\tau_\Lambda,
        \qquad
        T^{\tau,\mathfrak P}_\Lambda
        :=T^\tau_\Lambda-\Ctr^\tau_\Lambda
        =\Ret^\tau_\Lambda+B^\tau_\Lambda .
\end{equation}
Here \(\Ctr^\tau_\Lambda\) is subtracted and the retained term
\(\Ret^\tau_\Lambda\) satisfies a dyadic scalar Volterra bound.  The
definition permits a prescribed finite linear combination of scalar
diagonal branches, with the estimate imposed after that combination has
been formed.  Phase separation for distinct propagation speeds supplies a
concrete verification of this hypothesis.  This scalar analysis is the
deterministic counterpart of the four-leg centered estimate.  Model-dependent
diagonal classification and finite branch cancellation are encoded in the
explicit scalar Volterra hypothesis.

The mixed operators studied here are motivated by the two-speed
Klein--Gordon analysis in \cite{ZhaoColorPhase}.  The present paper isolates
the operator estimate and its cutoff convergence; the distinct-speed
diagonal calculation is included in Section~\ref{sec:verification}.
A subsequent companion paper \cite{ZhaoLocalized} develops a
continuous-frequency analogue of the centered four-leg estimate, with soft
incidence kernels on \(\mathbb R^d\).  The exact Fourier-diagonal contraction,
its deterministic Volterra resolution, and the frequency-envelope
formulation are specific to the present setting.

Section~\ref{sec:formulation} states the assumptions and the main theorem.
Section~\ref{sec:proof} contains the proof,
Section~\ref{sec:verification} treats the distinct-speed diagonal estimate,
and Appendix~\ref{subsec:envelope-variant} gives the frequency-envelope
extension.

\section{Setting and main theorem}\label{sec:formulation}

\subsection{Littlewood--Paley blocks and Bony products}\label{sec:setup}
Let \(P_N=\Delta_N\) be smooth dyadic Littlewood--Paley projectors on \(\T^d\), with \(N\in2^{\N_0}\); we use the standard notation and estimates of paradifferential calculus, see \cite{BCD}.  All Fourier variables lie in \(\Z^d\).  We use the shell count
\begin{equation}\label{eq:shell-count}
        \#\{n\in\Z^d: |n|\sim N\}\lesssim N^d
\end{equation}
and the standard inhomogeneous convention at low frequency.  Low-frequency blocks are absorbed into constants.  Fix a proportional aperture \(c_{\rm ap}\in(0,1)\), possibly small enough for the phase bounds used by the retained diagonal terms, and define
\begin{equation}\label{eq:bony}
        f<g=\sum_N S_{<c_{\rm ap}N}f\,P_Ng,
        \qquad
        f>g=g<f,
        \qquad
        f\circ g=\sum_{N\sim M}P_Nf\,P_Mg,
\end{equation}
where \(S_{<c_{\rm ap}N}=\sum_{Q<c_{\rm ap}N}P_Q\), with the usual smooth transition near the cutoff edge.  We write \(\chi_Q,\chi_M\) for the smooth Fourier cutoffs associated with \(P_Q,P_M\), \(\rho_N\) for the high stochastic-leg cutoff, and \(\chi_{\rm res}\) for the fixed smooth resonant cutoff in the channel \(|q+\ell|\sim |r|\).  Any fixed smooth cutoff with finite overlap is admissible.  Thus every paraproduct block satisfies
\begin{equation}\label{eq:aperture}
        Q\le c_{\rm ap}N
\end{equation}
up to constants fixed by the cutoff profile.  All implicit constants may depend on the selected aperture, but never on dyadic scales, time grids or Galerkin cutoffs.
Throughout the paper the symbol \(<\) refers to this fixed-aperture
low--high block.

We use the dyadic Galerkin cutoffs \(\Lambda\in2^{\N_0}\), with multiplier \(c_\Lambda(n)=\one_{|n|\le\Lambda}\).  Thus every cutoff operator is a finite Fourier operator, and each fixed dyadic block is eventually independent of \(\Lambda\).  When a variable is restricted to a fixed dyadic support, notation such as \(\ell_q^2\) or \(\ell_n^2\) denotes the corresponding finite-dimensional \(\ell^2\)-space.  The subscript records the frequency variable, not an additional weight.

For a dispersive channel \(i\), write
\begin{equation}\label{eq:duhamel}
        I_iF(t)=\int_0^t K_i(t-s,D)F(s)\dd s.
\end{equation}
The mixed block attached to a triple \((i;j,k)\) is
\begin{equation}\label{eq:block}
        T^{i;j,k}(w)=I_i(w<\Psi_j)\circ\Psi_k .
\end{equation}
The index \(i\) is the Duhamel channel; the indices \(j,k\) are the colors of the stochastic legs.

\subsection{Gaussian Hilbert space and Fourier-diagonal covariance}\label{subsec:gaussian-model}
Let \(\cC\) be a finite set of colors.  The Gaussian variables are realized on a real separable Hilbert space \(\mathfrak H\) with isonormal process \(\isoW:\mathfrak H\to L^2(\Omega)\).  We extend \(\isoW\) complex-linearly to the complexification \(\mathfrak H_{\mathbb C}\); its non-conjugated covariance is the complex-bilinear extension of the real inner product.  For each color \(a\), Fourier label \(n\in\Z^d\), and time \(t\in[0,T_0]\), write
\begin{equation}\label{eq:isonormal-coeff}
        \widehat\Psi_a(n,t)=\isoW(h_{a,n,t}),
        \qquad h_{a,n,t}\in\mathfrak H_\mathbb C,
\end{equation}
with the real-valued convention
\begin{equation}\label{eq:real-convention}
        \overline{\widehat\Psi_a(n,t)}=\widehat\Psi_a(-n,t).
\end{equation}
Thus the complex notation is shorthand for the corresponding real sine-cosine coordinates on each involution class \(\{n,-n\}\).  The variables indexed by \(n\) and \(-n\) are not independent; they are two complex encodings of the same real sine-cosine pair.  The covariance convention below is written in non-conjugated Fourier form, hence Fourier diagonality appears as the constraint \(n+m=0\).

Set
\[
        C_{ab}(n;t,s):=R_{ab}(n)\sigma_{ab}(n;t,s).
\]
The covariance is Fourier diagonal: for all \(a,b\in\cC\), all \(m,n\in\Z^d\), and all \(s,t\in[0,T_0]\),
\begin{equation}\label{eq:conv-cov}
        \E\bigl[\widehat\Psi_a(n,t)\widehat\Psi_b(m,s)\bigr]
        =\one_{n+m=0}C_{ab}(n;t,s).
\end{equation}
Symmetry of the complex-bilinear covariance and the real-valued convention
require
\begin{equation}\label{eq:cov-compat}
        C_{ba}(-n;s,t)=C_{ab}(n;t,s),
        \qquad
        C_{ab}(-n;t,s)=\overline{C_{ab}(n;t,s)}.
\end{equation}
For each involution class \(\{n,-n\}\), the induced covariance matrix on color and real sine-cosine coordinates is positive semidefinite.  We choose the factorization \(C_{ab}=R_{ab}\sigma_{ab}\) with \(R_{aa}(n)=1\).  For \(a\ne b\), assume
\begin{equation}\label{eq:cov-decay}
        |R_{ab}(n)|\le C_R\la n\ra^{-\kappa_{ab}},
        \qquad \kappa_{ab}\ge0.
\end{equation}
We set \(\kappa_{aa}=0\).  If \(R_{ab}\equiv0\) for \(a\ne b\), the corresponding cross-color diagonal is absent.

The spatial cutoff makes the Fourier sums finite.  The time-indexed Gaussian process is handled through the isonormal Hilbert space above.  Every probabilistic estimate below is first applied to a finite set of times and Fourier labels, hence to a finite Gaussian vector obtained from \(\isoW\); the continuous-time version is then obtained from the increment bounds in Assumption~\ref{ass:coeff} by the Banach-valued time-lift lemma.  The centered tensor estimate uses only the following coefficient bounds.

\begin{assumption}[Channel-dependent coefficient bounds]\label{ass:coeff}
Fix \(0<T_0\le1\).  To every Duhamel channel \(i\) assign a gain \(0<\lambda_i<\infty\), and to every stochastic color \(a\) assign a gain \(0<\alpha_a<\infty\).  For a mixed triple \((i;j,k)\), set
\begin{equation}\label{eq:Gamma-def}
        \Gamma_{i;j,k}:=\lambda_i+\alpha_j+\alpha_k,
        \qquad
        \chi_{i;j,k}:=\max\{\lambda_i,\alpha_j,\alpha_k\}.
\end{equation}
There exists \(\eta_0>0\) such that, for every \(0<\theta\le\eta_0\), for \(|n|\sim N\) and \(0\le t,t'\le T_0\),
\begin{align}
        |K_i(t,n)|&\le C_K N^{-\lambda_i},\label{eq:K-size}\\
        |K_i(t,n)-K_i(t',n)|&\le C_{K,\theta}|t-t'|^\theta N^{-\lambda_i+\lambda_i\theta},\label{eq:K-inc}\\
        \E|\wh\Psi_a(n,t)|^2&\le C_\Psi N^{-2\alpha_a},\label{eq:Psi-size}\\
        \E|\wh\Psi_a(n,t)-\wh\Psi_a(n,t')|^2&\le C_{\Psi,\theta}|t-t'|^{2\theta}N^{-2\alpha_a+2\alpha_a\theta}.
        \label{eq:Psi-inc}
\end{align}
The constants may depend on \(T_0\), \(\theta\) and the finite channel family, but not on the dyadic shell or on the cutoff.  The small exponent \(\theta\) is chosen after the strict dyadic loss \(\eps\) has been fixed, so that factors \(N^{\chi_{i;j,k}\theta}\) are absorbed by the available margins.  This power-law formulation is the main one used in Sections~\ref{sec:split-diag}--\ref{sec:assembly}.  Section~\ref{subsec:envelope-variant} gives the corresponding statement with general dyadic envelopes \(K_i(N)\) and \(A_a(N)\).
\end{assumption}

\subsection{Input and output spaces}
Let \(\cI(\cM)\) be the finite set of outer Duhamel labels appearing in the mixed family \(\cM\).  For \(0<T\le T_0\) and \(\sigma\in\R\), set
\begin{equation}\label{eq:X-space}
        X_T^\sigma:=C_TL^2\cap L_T^\infty B^\sigma_{2,\infty},
        \qquad
        \|w\|_{X_T^\sigma}:=\|w\|_{C_TL^2}+\|w\|_{L_T^\infty B^\sigma_{2,\infty}}.
\end{equation}
For a small fixed \(\eps>0\), define
\begin{equation}\label{eq:E-space}
\begin{aligned}
        E_{T,\cM}^{2,\sigma}:=\{w:&\ w\in X_T^\sigma,\\
        &\partial_tw\in \bigcap_{i\in\cI(\cM)}
          \bigl(C_TH^{-\lambda_i-\eps}\cap L_T^\infty B^{\sigma-\lambda_i}_{2,\infty}\bigr)\}.
\end{aligned}
\end{equation}
The norm is
\begin{equation}\label{eq:E-norm}
\begin{aligned}
        \|w\|_{E_{T,\cM}^{2,\sigma}}:=&\ \|w\|_{X_T^\sigma}\\
        &+\sum_{i\in\cI(\cM)}
        \left(\|\partial_tw\|_{C_TH^{-\lambda_i-\eps}}
        +\|\partial_tw\|_{L_T^\infty B^{\sigma-\lambda_i}_{2,\infty}}\right).
\end{aligned}
\end{equation}
For a single triple, this is the corresponding space with the single outer exponent \(\lambda_i\).  The derivative \(\partial_t w\) is understood in the distributional sense in time.  In the Volterra integration by parts used below, this means that, after applying each spatial dyadic projector, \(P_Qw(t)-P_Qw(0)=\int_0^t P_Q\partial_sw(s)\,\dd s\) in the indicated Sobolev--Besov topology.  The derivative component of this norm is used only by retained deterministic diagonals after this integration by parts.  The centered second-chaos estimates use the source space \(X_T^\sigma\).

\begin{definition}[Centered-admissible exponents]\label{def:admissible}
Given a mixed triple \((i;j,k)\) and \(\eps>0\), a pair \((s,\sigma)\) is \((i;j,k,\eps)\)-centered admissible if
\begin{equation}\label{eq:centered-window}
        \Gamma_{i;j,k}>\frac d2+10\eps,
        \qquad
        s<\lambda_i+\Gamma_{i;j,k}-d-10\eps,
        \qquad
        \max\{0,d-\Gamma_{i;j,k}\}+10\eps
        <\sigma<\lambda_i+\Gamma_{i;j,k}-d-10\eps .
\end{equation}
The underlying zero-loss \(\sigma\)-interval is non-empty precisely when
\begin{equation}\label{eq:nonempty-window}
        \max\{0,d-\Gamma_{i;j,k}\}<\lambda_i+\Gamma_{i;j,k}-d.
\end{equation}
For the \(\eps\)-window in \eqref{eq:centered-window}, the gap between
the two endpoints is required to exceed \(20\eps\); throughout the paper
\(\eps\) is chosen below the available strict margin.  In the singular subrange \(\Gamma_{i;j,k}\le d\), this is the stochastic summability condition
\begin{equation}\label{eq:nonempty-singular}
        2\Gamma_{i;j,k}+\lambda_i>2d.
\end{equation}
The first condition in \eqref{eq:centered-window} is the finite random-matrix summability requirement coming from the high-loop factor \(N^{d/2-\Gamma_{i;j,k}}\).  In the equal-gain case it is \(3\alpha>d/2\), which is weaker than the non-empty singular-window condition \(\alpha>2d/7\).  The gap comes from the operator summation after insertion of the low input.  The \(Q^{d/2}\)-branch forces the lower bound on \(\sigma\), while the \(M^{d/2}\)-branch forces the upper bound; the two inequalities overlap in the singular range exactly when \(2\Gamma+\lambda>2d\).
\end{definition}

Throughout, the finite family, aperture and target exponents are fixed
first.  We then choose \(\eps>0\) below all strict margins and finally choose
the time H\"older exponent \(\theta\) so that
\(\chi_\tau\theta<\eps/4\) for every \(\tau\in\cM\).  Constants may depend
on these data, but not on dyadic scales or Galerkin cutoffs.

\subsection{Deterministic diagonal resolutions}\label{subsec:contraction-resolution}
Let \(\tau=(i;j,k)\).  The finite Wick split in Section~\ref{sec:split-diag} produces a deterministic diagonal operator \(D^\tau_\Lambda\) and a centered second-chaos operator \(B^\tau_\Lambda\).  On dyadic shells, the diagonal kernel is
\begin{equation}\label{eq:D-formulation}
\begin{aligned}
        D_{\Lambda,N,Q}^{\tau}(q;t,s)
        :=&\ \chi_Q(q)
        \sum_\ell \rho_N(\ell)\rho_N(-\ell)\chi_{\rm res}(q+\ell,-\ell)\\
        &\times K_i(t-s,q+\ell)c_\Lambda(\ell)c_\Lambda(-\ell)
        R_{jk}(\ell)\sigma_{jk}(\ell;s,t),
\end{aligned}
\end{equation}
with \(D_{\Lambda,N,Q}^{\tau}=0\) when the covariance channel \((j,k)\) is absent.  It acts by
\begin{equation}\label{eq:D-action-formulation}
        \widehat{D^\tau_{\Lambda,N,Q}w}(q,t)
        =\int_0^tD^\tau_{\Lambda,N,Q}(q;t,s)\widehat w(q,s)\dd s .
\end{equation}
The full finite diagonal is the finite sum over \((N,Q)\).

For the target channel \(i\), define
\begin{equation}\label{eq:Y-space}
        Y_{T,i}^{s,\sigma}
        :=C_TH^{s-\lambda_i}\cap L_T^1B^{\sigma-\lambda_i}_{2,\infty}
\end{equation}
with the sum norm.

\begin{definition}[Admissible deterministic diagonal resolution]\label{def:contraction-resolution}
A deterministic diagonal resolution for a finite family \(\cM\) assigns to
each label \(\tau\in\cM\) deterministic diagonal Volterra operators
\(\Ret^\tau_\Lambda\) and \(\Ctr^\tau_\Lambda\) such that
\begin{equation}\label{eq:resolution}
        D^\tau_\Lambda=\Ret^\tau_\Lambda+\Ctr^\tau_\Lambda .
\end{equation}
The first term is retained and the second is subtracted from the raw mixed
block.  The resolution is dyadically admissible if, for each label with
\(\Ret^\tau_\Lambda\ne0\), there exists \(\delta_\tau>0\) such that
\begin{equation}\label{eq:retained-block-bound}
        \|\Ret^\tau_{\Lambda,N,Q}P_Qw\|_{C_TL^2}
        \le C_{\mathfrak P}N^{-\delta_\tau}
        \bigl(\|P_Qw\|_{L_T^\infty L^2}
        +\|P_Q\partial_tw\|_{L_T^\infty L^2}\bigr)
\end{equation}
for \(0<T\le T_0\) and \(Q\le c_{\rm ap}N\), uniformly in \(\Lambda\).
For every fixed \((N,Q)\), the corresponding two-time scalar kernels are
also required to converge in
\(C_{t,s}\cL(\ell_q^2,\ell_q^2)\).

The retained term may be any prescribed finite linear combination of scalar
diagonal branches, provided the combined kernel satisfies
\eqref{eq:retained-block-bound} and fixed-block convergence.  The raw
resolution is \(\Ret^\tau_\Lambda=D^\tau_\Lambda\), while the fully
Wick-subtracted resolution is \(\Ret^\tau_\Lambda=0\).
\end{definition}

The finite resolved cutoff operator is therefore
\begin{equation}\label{eq:resolved-cutoff}
        T^{\tau,\mathfrak P}_\Lambda
        :=T^\tau_\Lambda-\Ctr^\tau_\Lambda
        =\Ret^\tau_\Lambda+B^\tau_\Lambda .
\end{equation}
Every retained diagonal is a scalar Volterra multiplier in the low Fourier
label \(q\).  Assumption~\ref{ass:diag} gives a convenient phase-based
criterion for a raw branch; Proposition~\ref{prop:symbol-criterion} applies
directly to any prescribed finite combination.

\begin{assumption}[Raw Volterra contraction criterion]\label{ass:diag}
Consider a mixed triple \(\tau=(i;j,k)\) whose covariance channel is nonzero.  There exists a number
\begin{equation}\label{eq:diag-gain}
        \Delta_{\tau}>20\eps
\end{equation}
such that every dyadic diagonal summand
\begin{equation}\label{eq:diag-summand}
        K_i(t-s,q+\ell)R_{jk}(\ell)\sigma_{jk}(\ell;s,t),
        \qquad |\ell|\sim N,\quad |q|\le c_{\rm ap}N,
\end{equation}
can be decomposed into an oscillatory part and an absolutely summable remainder as follows.

\begin{enumerate}[label=(D\arabic*),leftmargin=2.2em]
\item There is a finite index set \(\mathfrak V_\tau\), independent of \(N,Q,\Lambda\), such that the oscillatory part is a finite sum
\begin{equation}\label{eq:osc-decomp}
        \sum_{\nu\in\mathfrak V_\tau}
        e^{\ii(t-s)\Phi_\nu(q,\ell)}a_\nu(q,\ell;t,s).
\end{equation}
For some \(\vartheta_\nu\ge0\),
\begin{equation}\label{eq:phase-lower}
        |\Phi_\nu(q,\ell)|\ge c_\Phi N^{\vartheta_\nu}
\end{equation}
whenever the branch is integrated by parts, and
\begin{equation}\label{eq:amp-bound}
        |a_\nu(q,\ell;t,s)|+|\partial_sa_\nu(q,\ell;t,s)|
        \le C_\Phi N^{-\Gamma_\tau-\kappa_{jk}}
\end{equation}
for \(0\le s\le t\le T_0\).  The loop gain satisfies
\begin{equation}\label{eq:osc-gain}
        \kappa_{jk}+\vartheta_\nu+\Gamma_\tau-d\ge \Delta_\tau .
\end{equation}
If \(\vartheta_\nu=0\), no phase denominator is used; then the branch is included only through the summability in \eqref{eq:osc-gain}.

\item The remainder \(e_{N,Q}(q,\ell;s,t)\) satisfies the loop bound
\begin{equation}\label{eq:remainder-loop}
        \sup_{0\le s\le t\le T_0}\sup_{|q|\le c_{\rm ap}N}
        \sum_{|\ell|\sim N}|e_{N,Q}(q,\ell;s,t)|
        \le C_\Phi N^{-\Delta_\tau}.
\end{equation}
\end{enumerate}

\end{assumption}

\begin{remark}[Phase gains in the criterion]\label{rem:phase-gains}
The exponent \(\vartheta_\nu\) is used only when the phase has a uniform
lower bound on the chosen aperture.  Otherwise we set
\(\vartheta_\nu=0\) and require summability from covariance decay,
absolute bounds, or a grouped scalar estimate.  Distinct-speed wave and
Klein--Gordon phases have \(\vartheta_\nu=1\); see
Section~\ref{sec:verification}.
\end{remark}

\begin{proposition}[Retained scalar diagonal criterion]\label{prop:symbol-criterion}
Let a deterministic diagonal resolution be chosen for one label
\(\tau=(i;j,k)\), and let \(\Ret_{\Lambda,N,Q}^\tau\) be the retained
diagonal Volterra block.  Suppose that its scalar kernel, after any
prescribed finite combinations have been formed, is a finite sum of
components of the following two types.
\begin{enumerate}[label=(\alph*),leftmargin=2.2em]
\item Absolute kernels \(r_{N,Q}(q;t,s)\) satisfying
\begin{equation}\label{eq:absolute-symbol-criterion}
        \sup_{0\le t\le T}\int_0^t
        \sup_{|q|\sim Q}|r_{N,Q}(q;t,s)|\,\dd s
        \le C N^{-\delta} .
\end{equation}
\item Oscillatory loop kernels of the form
\begin{equation}\label{eq:osc-loop-kernel}
        r_{N,Q}(q;t,s)
        =\sum_{|\ell|\sim N}\sum_{\nu\in\mathfrak V}
        e^{\ii(t-s)\Phi_\nu(q,\ell)}a_\nu(q,\ell;t,s),
\end{equation}
where \(\mathfrak V\) is finite, \(|\Phi_\nu(q,\ell)|>0\) on the retained support, and
\begin{equation}\label{eq:osc-symbol-criterion}
        \sup_{0\le s\le t\le T}\sup_{|q|\sim Q}
        \sum_{|\ell|\sim N}\sum_{\nu\in\mathfrak V}
        \frac{|a_\nu(q,\ell;t,s)|+|\partial_sa_\nu(q,\ell;t,s)|}{|\Phi_\nu(q,\ell)|}
        \le C N^{-\delta} .
\end{equation}
A post-loop phase \(\Phi(q)\) is the special case in which the \(\ell\)-sum has already been absorbed into \(a\).
\end{enumerate}
Then \(\Ret_{\Lambda,N,Q}^\tau\) satisfies the dyadic retained bound
\eqref{eq:retained-block-bound} with gain \(\delta\), uniformly in
\(\Lambda\).  The assertion applies directly to a prescribed finite sum of
scalar branches.
\end{proposition}

\begin{proof}
The absolute kernels are immediate from Minkowski's inequality.  For an oscillatory loop kernel, fix \(t\), \(q\), \(\ell\), and \(\nu\).  Write \(a_s=a_\nu(q,\ell;t,s)\) and \(\Phi=\Phi_\nu(q,\ell)\).  Integration by parts in the lower time variable gives
\begin{align*}
        \int_0^t e^{\ii(t-s)\Phi}a_s\widehat w(q,s)\,\dd s
        &=\frac{e^{\ii t\Phi}a_\nu(q,\ell;t,0)\widehat w(q,0)
        -a_\nu(q,\ell;t,t)\widehat w(q,t)}{\ii\Phi}\\
        &\quad +\int_0^t\frac{e^{\ii(t-s)\Phi}}{\ii\Phi}
        \bigl(\partial_sa_\nu(q,\ell;t,s)\widehat w(q,s)
        +a_\nu(q,\ell;t,s)\partial_s\widehat w(q,s)\bigr)\,\dd s .
\end{align*}
The estimate \eqref{eq:osc-symbol-criterion} is uniform at the boundary values \(s=0\) and \(s=t\).  After summing over \(\ell\) and the finite index set \(\mathfrak V\), the boundary terms, the amplitude-derivative term and the input-derivative term are bounded by
\[
        C N^{-\delta}
        \bigl(\|P_Qw\|_{L_T^\infty L^2}+\|P_Q\partial_tw\|_{L_T^\infty L^2}\bigr).
\]
This is \eqref{eq:retained-block-bound}.  Finite sums of absolute and oscillatory components only change the constant.
\end{proof}

\subsection{Main theorem}\label{sec:main}

The centered term and the deterministic diagonal are governed by separate
hypotheses.  The former uses the covariance and coefficient bounds; the
latter uses only Definition~\ref{def:contraction-resolution}.

Let \(\cM\) be a finite collection of triples \(\tau=(i;j,k)\).  For each cutoff \(\Lambda\in2^{\N_0}\), define
\begin{equation}\label{eq:finite-operator}
        T_{\Lambda}^{\tau}(w)=I_i(w<\Psi_{j,\Lambda})\circ\Psi_{k,\Lambda},
        \qquad \tau=(i;j,k)\in\cM,
\end{equation}
using the fixed Bony convention \eqref{eq:bony}.

\begin{theorem}[Mixed operator convergence with deterministic diagonal resolution]\label{thm:main}
Fix \(\eps>0\), \(0<T_0\le1\), a finite color set \(\cC\), a finite family \(\cM\) of mixed triples, and a fixed Bony aperture \(c_{\rm ap}\).  Assume the Fourier-diagonal time covariance model \eqref{eq:conv-cov}--\eqref{eq:cov-decay} and the coefficient bounds of Assumption~\ref{ass:coeff}.  Let \((s,\sigma)\) be \((\tau,\eps)\)-centered admissible in the sense of Definition~\ref{def:admissible} for every \(\tau\in\cM\).

Choose a dyadically admissible deterministic diagonal resolution
\(\mathfrak P=\{(\Ret^\tau_\Lambda,\Ctr^\tau_\Lambda)\}_{\tau\in\cM}\)
in the sense of Definition~\ref{def:contraction-resolution}.  Assume that
every nonzero retained diagonal has gain \(\delta_\tau>20\eps\) and that
\(s<\delta_\tau-10\eps\) for those labels.

For each \(\tau=(i;j,k)\), define the finite \(\mathfrak P\)-resolved operator
\begin{equation}\label{eq:ren-operator-general}
        T_{\Lambda}^{\tau,\mathfrak P}:=T_{\Lambda}^{\tau}-\Ctr^\tau_\Lambda .
\end{equation}
Then, for every \(0<T\le T_0\), the following assertions hold.
The full-probability event below may be chosen for \(T_0\) and therefore
works simultaneously for every \(T\le T_0\).

\begin{enumerate}[label=(\roman*),leftmargin=2.4em]
\item \textup{Finite Wick split.}  The cutoff operator admits the exact finite-cutoff Wick decomposition
\begin{equation}\label{eq:main-split}
        T_{\Lambda}^{\tau}=D_{\Lambda}^{\tau}+B_{\Lambda}^{\tau}.
\end{equation}
The diagonal part \(D_{\Lambda}^{\tau}\) is deterministic and supported by \(n=q\) in the external variables.  The centered part \(B_{\Lambda}^{\tau}\) is a finite sum of second homogeneous Gaussian chaoses and preserves the incidence
\begin{equation}\label{eq:external-incidence}
        n=q+\ell+r.
\end{equation}
Consequently,
\begin{equation}\label{eq:ren-split}
        T_{\Lambda}^{\tau,\mathfrak P}=\Ret^\tau_\Lambda+B^\tau_\Lambda.
\end{equation}

\item \textup{Centered operator component.}  For every finite \(p\ge2\), the centered operators \(B^\tau_\Lambda\) converge in \(L^p(\Omega)\) and almost surely in the centered operator topology
\begin{equation}\label{eq:centered-op-topology}
        \cL(X_T^\sigma,C_TH^{s-\lambda_i})
        \cap
        \cL(X_T^\sigma,L_T^1B^{\sigma-\lambda_i}_{2,\infty}).
\end{equation}
The limit \(B^\tau\) satisfies the operator-norm estimates
\begin{align}
        \|B^\tau\|_{L^p(\Omega;\cL(X_T^\sigma,C_TH^{s-\lambda_i}))}
        &\le C_p,\label{eq:centered-H}\\
        \|B^\tau\|_{L^p(\Omega;\cL(X_T^\sigma,L_T^1B^{\sigma-\lambda_i}_{2,\infty}))}
        &\le C_p.\label{eq:centered-B}
\end{align}
There is a full-probability event, independent of \(\tau\in\cM\), on which all centered cutoff sequences are Cauchy in the pathwise operator norm \eqref{eq:centered-op-topology}.  The centered estimate is an \(X_T^\sigma\)-estimate; the derivative components of \(E_{T,\cM}^{2,\sigma}\) are needed only when a retained diagonal is present.

\item \textup{Retained deterministic component.}  The retained diagonals converge to deterministic operators \(\Ret^\tau\) satisfying
\begin{equation}\label{eq:retained-main-bound}
        \|\Ret^\tau w\|_{C_TH^{s-\lambda_i}}
        +\|\Ret^\tau w\|_{L_T^1B^{\sigma-\lambda_i}_{2,\infty}}
        \le C\|w\|_{E_{T,\cM}^{2,\sigma}} .
\end{equation}

\item \textup{Assembled operator.}  On the same full-probability event,
\begin{equation}\label{eq:assembled-limit}
        T_{\Lambda}^{\tau,\mathfrak P}
        \longrightarrow
        T^{\tau,\mathfrak P}:=\Ret^\tau+B^\tau
\end{equation}
in the assembled topology
\begin{equation}\label{eq:assembled-op-topology}
        \cL(E_{T,\cM}^{2,\sigma},C_TH^{s-\lambda_i})
        \cap
        \cL(E_{T,\cM}^{2,\sigma},L_T^1B^{\sigma-\lambda_i}_{2,\infty}),
\end{equation}
for every \(\tau\in\cM\).  The convergence also holds in \(L^p(\Omega)\) for every finite \(p\ge2\).
\end{enumerate}
\end{theorem}

\begin{remark}[Raw and fully subtracted choices]\label{rem:raw-wick}
If \(\Ctr^\tau_\Lambda=0\), the theorem gives the raw limit \(T^\tau=\Ret^\tau+B^\tau\), provided the deterministic contraction \(D^\tau_\Lambda=\Ret^\tau_\Lambda\) is admissible.  If \(\Ret^\tau_\Lambda=0\), then \(\Ctr^\tau_\Lambda=D^\tau_\Lambda\) and
\[
        T_{\Lambda}^{\tau,\mathfrak P}=T^\tau_\Lambda-D^\tau_\Lambda=B^\tau_\Lambda,
\]
so the theorem constructs the Wick-centered branch.  The same conclusion
holds in the one-color case, using the second-chaos decoupling estimate in
Lemma~\ref{lem:local-normal-form}.
\end{remark}

\begin{remark}[Equal-gain window]\label{rem:thresholds}
If all three gains in a branch equal \(\alpha\), then \(\Gamma_\tau=3\alpha\), and Definition~\ref{def:admissible} gives
\[
        s<4\alpha-d,
        \qquad
        \max\{0,d-3\alpha\}<\sigma<4\alpha-d
\]
with strict losses.  In the singular range \(3\alpha\le d\), the lower and upper \(\sigma\)-bounds overlap iff \(d-3\alpha<4\alpha-d\), equivalently \(\alpha>2d/7\).
\end{remark}

\section{Proof of the theorem}\label{sec:proof}

\subsection{Finite Wick split and deterministic diagonals}\label{sec:split-diag}

All Wick and tensor estimates below are first applied to finite Fourier and
time skeletons.  Lemma~\ref{lem:time-lift} then passes to continuous time.

\paragraph{Dyadic kernel formula.}
Fix a dyadic triple \((N,Q,M)\) with \(Q\le c_{\rm ap}N\).  The kernel of the block \(P_MT_{\Lambda}^{i;j,k}(P_Qw)\) is
\begin{equation}\label{eq:dyadic-kernel}
\begin{aligned}
        A_{\Lambda,N,Q,M}^{i;j,k}(n,q;t,s)
        :=&\ \chi_M(n)\chi_Q(q)
        \sum_{\ell+r=n-q}\rho_N(\ell)\rho_N(r)\chi_{\rm res}(q+\ell,r)\\
        &\times K_i(t-s,q+\ell)
        \wh\Psi_{j,\Lambda}(\ell,s)\wh\Psi_{k,\Lambda}(r,t).
\end{aligned}
\end{equation}
The operator acts by
\begin{equation}\label{eq:kernel-action}
        P_MT_{\Lambda}^{i;j,k}(P_Qw)(n,t)
        =\int_0^t\sum_q A_{\Lambda,N,Q,M}^{i;j,k}(n,q;t,s)\wh w(q,s)\dd s.
\end{equation}
The support conditions imply
\begin{equation}\label{eq:support-geometry}
        |q|\sim Q,
        \qquad
        |\ell|\sim |r|\sim N,
        \qquad
        |n|\sim M,
        \qquad
        M\lesssim N.
\end{equation}
The case \(M\ll N\) is permitted and corresponds to high-high cancellation in \(q+\ell+r=n\).

\paragraph{The finite Wick identity.}
For fixed colors, modes and times, set \(X=\wh\Psi_{j,\Lambda}(\ell,s)\) and \(Y=\wh\Psi_{k,\Lambda}(r,t)\).  The second-order Wick product is defined in the Gaussian Hilbert space by
\begin{equation}\label{eq:finite-wick}
        XY=\wick{XY}+
        \E[XY].
\end{equation}
For every finite list of modes and time variables this is the ordinary finite-dimensional Wick identity on the span of the corresponding kernels.
The Fourier-diagonal covariance gives
\begin{equation}\label{eq:finite-cov}
        \E\bigl[\wh\Psi_{j,\Lambda}(\ell,s)\wh\Psi_{k,\Lambda}(r,t)\bigr]
        =\one_{\ell+r=0}c_\Lambda(\ell)c_\Lambda(-\ell)R_{jk}(\ell)\sigma_{jk}(\ell;s,t).
\end{equation}
Inserting \eqref{eq:finite-wick} into \eqref{eq:dyadic-kernel} yields
\begin{equation}\label{eq:dyadic-split}
        A_{\Lambda,N,Q,M}^{i;j,k}=B_{\Lambda,N,Q,M}^{i;j,k}+
        \one_{n=q}\chi_M(q)D_{\Lambda,N,Q}^{i;j,k}(q;t,s),
\end{equation}
where \(D=0\) if the covariance channel is absent, and otherwise
\begin{equation}\label{eq:D-kernel}
\begin{aligned}
        D_{\Lambda,N,Q}^{i;j,k}(q;t,s)
        :=&\ \chi_Q(q)
        \sum_\ell \rho_N(\ell)\rho_N(-\ell)
        \chi_{\rm res}(q+\ell,-\ell)\\
        &\times K_i(t-s,q+\ell)c_\Lambda(\ell)c_\Lambda(-\ell)
        R_{jk}(\ell)\sigma_{jk}(\ell;s,t).
\end{aligned}
\end{equation}
The relation \(\ell+r=0\) forces \(n=q\).  The factor \(\chi_M(q)\) in \eqref{eq:dyadic-split} is the output projection inherited from \(P_M\); equivalently, the full diagonal operator is \(D_{\Lambda,N,Q}^{i;j,k}\) and its dyadic output block is \(P_MD_{\Lambda,N,Q}^{i;j,k}P_Q\).  Hence the lower-chaos branch is a scalar Fourier multiplier in the low input variable.  The centered part \(B\) is still supported by \(n=q+\ell+r\), including the centered Wick-square sector with \(r=-\ell\).

\begin{lemma}[Finite Wick split]\label{lem:color-split}
For every cutoff, dyadic triple and mixed label \((i;j,k)\), \eqref{eq:dyadic-split} is an exact algebraic identity.  The branch \(D\) is deterministic and Fourier diagonal in \((q,n)\); the branch \(B\) is a finite sum of second homogeneous centered Gaussian chaoses.
\end{lemma}

\begin{proof}
Fix the dyadic block and the two times \((s,t)\).  The family of Fourier
coefficients appearing in \eqref{eq:dyadic-kernel} is then a finite Gaussian
vector in the isonormal space \(\mathfrak H\).  For each pair \(X,Y\) in this
finite vector, the degree-two Wick projection is
\(\wick{XY}:=XY-\E[XY]\); this is the elementary finite-vector form of
Gaussian Hilbert space and Wiener chaos calculus
\cite{Janson,Nualart,PeccatiTaqqu}.  Formula \eqref{eq:finite-cov} follows
from Fourier diagonality of \eqref{eq:conv-cov}.  Substituting this identity
into \eqref{eq:dyadic-kernel} gives \eqref{eq:dyadic-split}.  The indicator
\(\ell+r=0\) turns the external incidence \(n=q+\ell+r\) into \(n=q\),
proving the diagonal statement.  After subtracting the scalar expectation,
the remaining monomials are centered degree-two Gaussian chaoses.  The
continuous-time identity follows by applying this finite-time argument on
rational time skeletons and using the coefficient continuity in
\(L^2(\Omega)\).
\end{proof}

\paragraph{Finite assembly.}
Let \(\mathfrak P\) be a deterministic diagonal resolution, so that
\(D^\tau_\Lambda=\Ret^\tau_\Lambda+\Ctr^\tau_\Lambda\).  The Wick split
immediately gives, block by block,
\begin{equation}\label{eq:finite-ren-equals-B}
        P_MT_{\Lambda}^{\tau,\mathfrak P}P_Q
        =P_M\Ret^\tau_\Lambda P_Q+P_MB^\tau_\Lambda P_Q .
\end{equation}
If \(\Ret^\tau_\Lambda=0\), then
\(T_{\Lambda}^{\tau,\mathfrak P}=B^\tau_\Lambda\) is a centered
second-chaos operator.

\paragraph{Diagonal Volterra estimate.}
We prove the deterministic Volterra criterion for retained diagonals.

\begin{proposition}[Volterra criterion for an admissible diagonal]\label{prop:diagonal}
Assume Assumption~\ref{ass:diag} for \(\tau=(i;j,k)\), and let \(s<\Delta_\tau-10\eps\).  Then the choice \(\Ret^\tau_\Lambda=D^\tau_\Lambda\), \(\Ctr^\tau_\Lambda=0\) is a dyadically admissible deterministic diagonal resolution for this label with any retained gain \(\delta_\tau<\Delta_\tau\).  The cutoff diagonals converge to a deterministic operator \(D^\tau\), and \eqref{eq:retained-main-bound} holds with \(\Ret^\tau=D^\tau\).
\end{proposition}

\begin{proof}
We suppress the mixed label and write \(\Delta=\Delta_\tau\), \(\lambda=\lambda_i\), and \(\Gamma=\Gamma_\tau\).

For the absolutely summable remainder in Assumption~\ref{ass:diag}, the diagonal kernel is already bounded after the high loop:
\begin{equation}\label{eq:D-rem-mult}
        \sup_{q,t,s}\left|\sum_{|\ell|\sim N}e_{N,Q}(q,\ell;s,t)\right|
        \lesssim N^{-\Delta}.
\end{equation}
It therefore acts on the input mode \(q\) by a scalar multiplier of size \(N^{-\Delta}\).  We next treat one oscillatory summand; the number of such summands is fixed independently of the dyadic scales.  For fixed \(q,\ell,t\), set \(a_s=a(q,\ell;t,s)\) and \(\Phi=\Phi(q,\ell)\).  The contribution to the Volterra operator is
\begin{equation}\label{eq:D-osc-one}
        \mathcal V_{q,\ell}(t)
        =\int_0^t e^{\ii(t-s)\Phi}a_s\widehat w(q,s)\dd s .
\end{equation}
If the branch has \(\vartheta>0\), integration by parts in \(s\) gives the identity
\begin{equation}\label{eq:D-IBP-full}
\begin{aligned}
        \mathcal V_{q,\ell}(t)
        =&\ \frac{e^{\ii t\Phi}a(q,\ell;t,0)\widehat w(q,0)
        -a(q,\ell;t,t)\widehat w(q,t)}{\ii\Phi} \\
        &+\int_0^t\frac{e^{\ii(t-s)\Phi}}{\ii\Phi}
        \bigl(\partial_sa(q,\ell;t,s)\widehat w(q,s)
        +a(q,\ell;t,s)\partial_s\widehat w(q,s)\bigr)\dd s .
\end{aligned}
\end{equation}
Here \(t\) is treated as an external parameter; no derivative in \(t\) is used.  By \eqref{eq:amp-bound} and \eqref{eq:phase-lower}, the summand is bounded by
\begin{equation}\label{eq:D-one-summand}
        N^{-\Gamma-\kappa_{jk}-\vartheta}
        \Bigl(\|P_Qw\|_{L_T^\infty L^2}
        +\|P_Q\partial_tw\|_{L_T^\infty L^2}\Bigr)
\end{equation}
at the level of the scalar multiplier acting on the \(q\)-shell.  Summing the high loop, which has \(O(N^d)\) lattice points, gives
\begin{equation}\label{eq:D-shell-size-refined}
        N^dN^{-\Gamma-\kappa_{jk}-\vartheta}
        \lesssim N^{-\Delta},
\end{equation}
by \eqref{eq:osc-gain}.  If \(\vartheta=0\), no integration by parts is used.  Instead,
\[
        \sum_{|\ell|\sim N}|a(q,\ell;t,s)|
        \lesssim N^{d-\Gamma-\kappa_{jk}}
        \le N^{-\Delta}
\]
directly from \eqref{eq:amp-bound} and \eqref{eq:osc-gain}.  Consequently,
for every dyadic diagonal block,
\begin{equation}\label{eq:D-dyadic-L2-refined}
        \|D_{N,Q}P_Qw\|_{C_TL^2}
        \lesssim
        N^{-\Delta}\|P_Qw\|_{L_T^\infty L^2}
        +N^{-\Delta}\|P_Q\partial_tw\|_{L_T^\infty L^2},
\end{equation}
with the derivative term omitted for purely absolute branches.

We now sum this estimate.  Because the diagonal branch is supported by \(n=q\), the output frequency is \(Q\), not \(N\).  For the Sobolev target,
\begin{equation}\label{eq:D-H-weighted}
\begin{aligned}
        \|D_{N,Q}P_Qw\|_{C_TH^{s-\lambda}}
        &\lesssim Q^{s-\lambda}N^{-\Delta}\|P_Qw\|_{L_T^\infty L^2} \\
        &\quad +Q^{s-\lambda}N^{-\Delta}\|P_Q\partial_tw\|_{L_T^\infty L^2}.
\end{aligned}
\end{equation}
By Cauchy--Schwarz in the dyadic \(Q\)-sum, with the borderline logarithm
absorbed into \(N^\eps\),
\begin{equation}\label{eq:D-first-sum}
        \sum_{Q\le c_{\rm ap}N}Q^{s-\lambda}\|P_Qw\|_{L^2}
        \lesssim_\eps N^{(s-\lambda)_++\eps}\|w\|_{L^2}.
\end{equation}
The corresponding high-shell factor is \(N^{-\Delta+(s-\lambda)_++\eps}\), which is summable under \(s<\Delta-10\eps\) because either \(s\le\lambda\), in which case \((s-\lambda)_+=0\), or \(s>\lambda\), in which case \((s-\lambda)_+<s<\Delta\).  The derivative term uses
\begin{equation}\label{eq:D-dt-weighted}
        \|P_Q\partial_tw\|_{L_T^\infty L^2}
        \le Q^{\lambda+\eps}\|\partial_tw\|_{C_TH^{-\lambda-\eps}},
\end{equation}
so that
\begin{equation}\label{eq:D-der-sum}
        \sum_{Q\le c_{\rm ap}N}Q^{s-\lambda}N^{-\Delta}
        \|P_Q\partial_tw\|_{L_T^\infty L^2}
        \lesssim N^{-\Delta+(s)_++2\eps}
        \|\partial_tw\|_{C_TH^{-\lambda-\eps}}.
\end{equation}
The high-shell sum is finite because \(s<\Delta-10\eps\); if \(s<0\), the positive part is absent.  This proves the \(C_TH^{s-\lambda}\) diagonal bound.

For the Besov target, fix the output/input block \(Q\).  From \eqref{eq:D-dyadic-L2-refined},
\begin{equation}\label{eq:D-Besov-refined}
\begin{aligned}
        Q^{\sigma-\lambda}\|D_{N,Q}P_Qw\|_{L_T^\infty L^2}
        &\lesssim N^{-\Delta}Q^{-\lambda}
        \|w\|_{L_T^\infty B^\sigma_{2,\infty}} \\
        &\quad +N^{-\Delta}
        \|\partial_tw\|_{L_T^\infty B^{\sigma-\lambda}_{2,\infty}}.
\end{aligned}
\end{equation}
The sum over \(N\gtrsim Q\) is finite because \(\Delta>0\), uniformly in \(Q\).  Hence
\[
        \|D w\|_{L_T^\infty B^{\sigma-\lambda}_{2,\infty}}
        \lesssim \|w\|_{E_{T,\cM}^{2,\sigma}}.
\]
The \(L_T^1\) estimate follows immediately with the factor \(T\).

For the cutoff limit, fix a dyadic pair \((N,Q)\).  The diagonal matrix is finite and its entries converge because the cutoff symbols converge pointwise.  The estimates above give a summable majorant independent of the cutoff.  Lemma~\ref{lem:dyadic-tail-assembly} therefore proves convergence in the same operator norms.  This proves the proposition.
\end{proof}

\subsection{Centered second-chaos estimate}\label{sec:centered}

\paragraph{Local realification and Wick-square decoupling.}\label{subsec:local-whitening}
Choose once and for all a half-lattice \(\Z^d_+\subset\Z^d\) containing one
representative of every involution class \([m]=\{m,-m\}\), with \(0\in
\Z^d_+\).  For \(m\ne0\), write
\[
        m=\epsilon(m)\underline m,
        \qquad \underline m\in\Z^d_+,
        \qquad \epsilon(m)\in\{+1,-1\},
\]
and set \(\underline{0}=0\), \(\epsilon(0)=1\).  The normalization below is
local to one fixed-time evaluation or one increment comparison; an entire
time grid is never whitened at once.  Hence only a uniformly bounded number
of first-Wiener kernels occurs in each involution class, independently of
the rational grid used in the time lift.

\begin{lemma}[Local normalization and decoupled normal form]\label{lem:local-normal-form}
Fix a cutoff and a dyadic block \((N,Q,M)\).  Consider either one value of the
centered kernel, one increment between two Volterra pairs, or one member of
the finite cutoff family used in the tail argument.  In an increment
estimate, collect jointly all first-Wiener kernels at the two endpoints and
in the finite telescope connecting them.  There exists an integer \(d_0\),
depending only on the finite color family and on the number of telescope
terms, with the following properties.

For every \(\underline m\in\Z^d_+\), there is a standard real Gaussian vector
\[
        G_{\underline m}
        =(G_{\underline m,\mu})_{1\le\mu\le d_0},
\]
and these vectors are independent for distinct representatives.  With
\[
        a=(\underline\ell,\mu),\qquad
        b=(\underline r,\lambda),\qquad c=q,\qquad e=n,
\]
the centered kernel has a finite expansion
\begin{equation}\label{eq:centered-normal-form}
        B_{\Lambda,N,Q,M}^{i;j,k}(e,c;t,s)
        =\sum_{\nu}\sum_{a,b}
        H_{\nu;a,b,c,e}(t,s)
        \wick{G_{\underline\ell,\mu}G_{\underline r,\lambda}} .
\end{equation}
The pattern index \(\nu\) contains the two Fourier signs, together with the
finite sine--cosine, color and cutoff choices.  For fixed \(\nu\), there are
signs \(\epsilon_{\nu,1},\epsilon_{\nu,2}\in\{+1,-1\}\) such that the
coefficient tensor is supported by
\begin{equation}\label{eq:H-support}
        n=q+\epsilon_{\nu,1}\underline\ell
              +\epsilon_{\nu,2}\underline r,
        \qquad |\underline\ell|\sim|\underline r|\sim N,
        \qquad |q|\sim Q,
        \qquad |n|\sim M,
        \qquad Q\le c_{\rm ap}N,
\end{equation}
up to the fixed Littlewood--Paley overlap, and
\begin{equation}\label{eq:H-size}
        |H_{\nu;a,b,c,e}(t,s)|
        \lesssim N^{-\Gamma_{i;j,k}}.
\end{equation}
For an increment tensor, the right-hand side carries the additional
increment factor from Assumption~\ref{ass:coeff}.  The number of patterns
and the auxiliary dimension \(d_0\) are uniform in the dyadic scales,
cutoff and time points.

For every finite-dimensional Banach space \(X\), every finite coefficient
family \((C_{a,b})\subset X\), and every \(p\ge2\), write
\(G_a=G_{\underline\ell,\mu}\) for \(a=(\underline\ell,\mu)\),
and similarly for \(G_b\).  Then
\begin{equation}\label{eq:local-decoupling}
        \left\|\sum_{a,b}C_{a,b}\wick{G_aG_b}
        \right\|_{L^p(\Omega;X)}
        \le C_{\rm dec}
        \left\|\sum_{a,b}C_{a,b}G_a\widetilde G_b
        \right\|_{L^p(\Omega\times\widetilde\Omega;X)},
\end{equation}
where \((\widetilde G_b)\) is an independent copy.  The constant
\(C_{\rm dec}\) depends only on the chaos order.
\end{lemma}

\begin{proof}
For \(\underline m\in\Z^d_+\), let \(\mathfrak H_{\underline m}\) be the real
closed subspace generated by the sine--cosine representatives of all
kernels with signed labels \(\pm\underline m\) that occur in the selected
value or increment comparison.  If \(\underline m\ne\underline m'\),
Fourier diagonality gives zero covariance between every generator of
\(\mathfrak H_{\underline m}\) and every generator of
\(\mathfrak H_{\underline m'}\).  The two subspaces are therefore
orthogonal.  Since the isonormal process is Gaussian, its restrictions to
these subspaces are independent.

Factor out the dyadic size before performing any finite-dimensional linear
algebra.  On \(|m|\sim N\), write
\[
        h_{a,m,t}=N^{-\alpha_a}\bar h_{a,m,t},
        \qquad \|\bar h_{a,m,t}\|_{\mathfrak H}\lesssim1,
\]
and normalize increments with the additional factor
\(|t-t'|^\theta N^{\alpha_a\theta}\) from \eqref{eq:Psi-inc}.  In each
class, the collected normalized kernels form a family of cardinality at
most \(d_0\).  Its covariance matrix \(\Sigma_{\underline m}\) satisfies
\[
        \|\Sigma_{\underline m}\|_{\op}
        \le\operatorname{tr}\Sigma_{\underline m}\lesssim d_0.
\]
Choose a factorization
\(\Sigma_{\underline m}=U_{\underline m}U_{\underline m}^*\) on its range
and pad by zero coordinates to dimension \(d_0\).  Every collected
first-Wiener variable then has the form
\[
        \isoW(\bar h_\rho)
        =\sum_{\mu=1}^{d_0}
          U_{\underline m}(\rho,\mu)G_{\underline m,\mu},
        \qquad \|U_{\underline m}\|_{\op}\lesssim d_0^{1/2}.
\]
For an increment comparison, the same factorization contains both
endpoints, so subtraction is performed in a common Gaussian coordinate
system.

The signed Fourier labels are not used as Gaussian indices.  Instead, their
signs are recorded in \(\nu\), while \(\underline\ell\) and
\(\underline r\) index the standard Gaussian coordinate family attached
to the involution classes.  Distinct representatives give independent
vectors; coincident representatives are handled by the Wick product and the
order-two decoupling.  Expanding the synthesis maps therefore gives
\eqref{eq:centered-normal-form}; the physical convolution relation becomes
\eqref{eq:H-support}.  The Duhamel multiplier contributes
\(N^{-\lambda_i}\), and the two first-Wiener kernels contribute
\(N^{-\alpha_j}\) and \(N^{-\alpha_k}\), proving
\eqref{eq:H-size}.  The finite telescope for a time difference places
exactly one increment in each summand and gives the stated increment bound.

Finally, a centered quadratic form in the common coordinates is
\[
        \sum_{a,b}C_{a,b}
        \bigl(G_aG_b-\E[G_aG_b]\bigr).
\]
Its coefficient array may be symmetrized, and the Banach-valued decoupling
theorem for homogeneous Gaussian chaoses of order two gives
\eqref{eq:local-decoupling}; see \cite{DeLaPenaGine}.  The scalar covariance
has already been assigned to the deterministic diagonal by the finite Wick
split.
\end{proof}

\begin{remark}[Patternwise tensor estimate]\label{rem:patternwise-tensor}
No independence between different values of \(\nu\) is needed.  For each
fixed sign and realification pattern, apply \eqref{eq:local-decoupling} and
Proposition~\ref{prop:tensor}, and then sum the finitely many bounds.  Since
changing the two signs is a bijection on the relevant lattice shells, the
incidence counts below are uniform in \(\nu\).
\end{remark}

\paragraph{A four-legged second-chaos tensor estimate.}
The next proposition is the second-order finite random tensor estimate in
the form needed below; compare \cite{DNYTensor,Kaneshiro}.  Its four
flattenings correspond to the four choices of placing each of the two
Gaussian legs on the input or output side.

\begin{proposition}[Finite second-order tensor estimate]\label{prop:tensor}
Let \(A,B,C,E\) be finite sets and set
\[
        \Lfin=2+|A|+|B|+|C|+|E|.
\]
Let \((g_a)_{a\in A}\) and \((h_b)_{b\in B}\) be independent standard real Gaussian families.  For a deterministic tensor \(H=(H_{a,b,c,e})\), define the random matrix
\begin{equation}\label{eq:tensor-Z}
        Z_{e,c}=\sum_{a\in A}\sum_{b\in B}H_{a,b,c,e}g_ah_b .
\end{equation}
Define the four-cut norm
\begin{equation}\label{eq:four-flattenings}
\begin{aligned}
\mathfrak R_4(H):=\max\bigl\{&
  \|H\|_{\ell^2_{a,b,c}\to\ell^2_e},
  \ \|H\|_{\ell^2_{a,b,e}\to\ell^2_c},\\
 &\|H\|_{\ell^2_{a,c}\to\ell^2_{b,e}},
  \ \|H\|_{\ell^2_{b,c}\to\ell^2_{a,e}}
\bigr\}.
\end{aligned}
\end{equation}
Then, for every \(p\ge2\),
\begin{equation}\label{eq:tensor-bound}
        \bigl\|\|Z\|_{\ell^2_c\to\ell^2_e}\bigr\|_{L^p(\Omega)}
        \le C\bigl(p+\log\Lfin\bigr)\mathfrak R_4(H),
\end{equation}
where \(C\) is an absolute constant.  The same bound holds for complex Gaussian families after realification and, after the order-two decoupling in Lemma~\ref{lem:local-normal-form}, for arbitrary finite centered quadratic Gaussian chaoses.  It is also stable under the finite sign, real-Fourier and auxiliary-coordinate patterns produced there: each pattern is estimated separately, and a bounded enlargement of the Gaussian legs or a finite sum of pattern bounds changes only the implicit constant.
\end{proposition}

\begin{proof}
We give the proof in real coordinates.  Writing complex coordinates in real form only multiplies the finite index sets by an absolute factor and agrees with the convention of Lemma~\ref{lem:local-normal-form}.  Choose
\[
        s_0=\left\lceil C_0\bigl(p+\log\Lfin\bigr)\right\rceil
\]
with \(C_0\) large.  Then \(s_0\ge p\), and it is enough to estimate
\(\|Z\|_{L^{s_0}(\Omega;S_{s_0})}\), because
\(\|Z\|_{\op}\le \|Z\|_{S_{s_0}}\).  The Schatten spaces below act on one of the one- or two-factor spaces
\[
        \ell^2_c,
        \quad \ell^2_e,
        \quad \ell^2_b\otimes\ell^2_c,
        \quad \ell^2_b\otimes\ell^2_e,
        \quad \ell^2_a\otimes\ell^2_c,
        \quad \ell^2_a\otimes\ell^2_e .
\]
Their dimensions are at most \(\Lfin^2\).  Hence every passage from an operator norm to an \(S_{s_0}\)-norm costs at most \(\Lfin^{2/s_0}\), which is uniformly bounded by the choice of \(s_0\).

We use the rectangular non-commutative Khintchine inequality \cite{LustPisier}: for finite matrices \(A_b:H_1\to H_2\) and independent standard Gaussians \((\gamma_b)\),
\begin{equation}\label{eq:ncK-rectangular}
\left\|\sum_b\gamma_bA_b\right\|_{L^{s_0}(\Omega;S_{s_0}(H_1,H_2))}
\lesssim \sqrt {s_0}\left(
\left\|\Big(\sum_bA_bA_b^*\Big)^{1/2}\right\|_{S_{s_0}(H_2)}
+
\left\|\Big(\sum_bA_b^*A_b\Big)^{1/2}\right\|_{S_{s_0}(H_1)}
\right).
\end{equation}
The displayed rectangular form follows from the usual self-adjoint form by placing \(A_b\) in the off-diagonal block of \(H_1\oplus H_2\); this is the standard self-adjoint block trick.

Condition first on the \(g\)-variables and write
\begin{equation}\label{eq:tensor-condition-h}
        Z=\sum_{b\in B}h_bS_b(g),
        \qquad
        S_b(g)_{e,c}=\sum_{a\in A}H_{a,b,c,e}g_a .
\end{equation}
Applying \eqref{eq:ncK-rectangular} to the Gaussian sum of operators \(S_b(g):\ell^2_c\to\ell^2_e\) gives
\begin{equation}\label{eq:ncK-first-strict}
\begin{aligned}
        \|Z\|_{L_h^{s_0}S_{s_0}}
        \lesssim \sqrt {s_0}\Bigg(&
        \left\|\Big(\sum_b S_b(g)S_b(g)^*\Big)^{1/2}\right\|_{S_{s_0}(\ell^2_e)}\\
        &+
        \left\|\Big(\sum_b S_b(g)^*S_b(g)\Big)^{1/2}\right\|_{S_{s_0}(\ell^2_c)}\Bigg).
\end{aligned}
\end{equation}
We estimate the two square functions in \(L_g^{s_0}\).

For the first square function define
\[
        \mathcal S(g):\ell^2_b\otimes\ell^2_c\longrightarrow\ell^2_e,
        \qquad
        \mathcal S(g)_{e,(b,c)}=S_b(g)_{e,c}.
\]
Then, for every Schatten exponent \(r\ge2\),
\[
        \|\mathcal S(g)\|_{S_r}
        =\left\|\Big(\sum_bS_b(g)S_b(g)^*\Big)^{1/2}\right\|_{S_r}.
\]
Write \(\mathcal S(g)=\sum_ag_aT_a\), where \((T_a)_{e,(b,c)}=H_{a,b,c,e}\).  A second application of \eqref{eq:ncK-rectangular} yields
\begin{equation}\label{eq:first-square-bound}
\begin{aligned}
        \|\mathcal S(g)\|_{L_g^{s_0}S_{s_0}}
        &\lesssim \sqrt {s_0}\left(
        \left\|\Big(\sum_aT_aT_a^*\Big)^{1/2}\right\|_{S_{s_0}}
        +
        \left\|\Big(\sum_aT_a^*T_a\Big)^{1/2}\right\|_{S_{s_0}}
        \right)\\
        &\lesssim \sqrt {s_0}\left(
        \|H\|_{\ell^2_{a,b,c}\to\ell^2_e}
        +
        \|H\|_{\ell^2_{b,c}\to\ell^2_{a,e}}
        \right).
\end{aligned}
\end{equation}
Indeed, \(\|\sum_aT_aT_a^*\|_{\op}^{1/2}\) is the norm of the flattening \((a,b,c)\to e\), while \(\|\sum_aT_a^*T_a\|_{\op}^{1/2}\) is the norm of \((b,c)\to(a,e)\).

For the second square function define
\[
        \mathcal T(g):\ell^2_c\longrightarrow\ell^2_b\otimes\ell^2_e,
        \qquad
        \mathcal T(g)_{(b,e),c}=S_b(g)_{e,c}.
\]
Similarly,
\[
        \|\mathcal T(g)\|_{S_r}
        =\left\|\Big(\sum_bS_b(g)^*S_b(g)\Big)^{1/2}\right\|_{S_r}.
\]
With \(\mathcal T(g)=\sum_ag_aT'_a\), the same argument gives
\begin{equation}\label{eq:second-square-bound}
        \|\mathcal T(g)\|_{L_g^{s_0}S_{s_0}}
        \lesssim \sqrt {s_0}\left(
        \|H\|_{\ell^2_{a,c}\to\ell^2_{b,e}}
        +
        \|H\|_{\ell^2_c\to\ell^2_{a,b,e}}
        \right).
\end{equation}
The last norm is the adjoint of the flattening
\((a,b,e)\to c\).  Thus the two square functions produce exactly the four
norms in \eqref{eq:four-flattenings}; reversing the order of conditioning
only interchanges the two mixed cuts.  Consequently
\[
\begin{aligned}
        \|Z\|_{L_{g,h}^{s_0}S_{s_0}}
        &\lesssim \sqrt{s_0}\bigl(
        \|\mathcal S(g)\|_{L_g^{s_0}S_{s_0}}
        +\|\mathcal T(g)\|_{L_g^{s_0}S_{s_0}}\bigr)\\
        &\lesssim s_0\mathfrak R_4(H).
\end{aligned}
\]
This proves \eqref{eq:tensor-bound}, since
\(s_0\lesssim p+\log\Lfin\).

For same-family Wick squares, the decoupling inequality for homogeneous Gaussian chaoses with values in \(X=\cL(\ell_c^2,\ell_e^2)\) gives
\begin{equation}\label{eq:decouple-wick}
        \left\|\sum_{a,b}C_{a,b}\wick{g_ag_b}\right\|_{L^p(\Omega;X)}
        \lesssim
        \left\|\sum_{a,b}C_{a,b}g_a h_b\right\|_{L^p(\Omega_g\times\Omega_h;X)},
\end{equation}
where \((h_b)\) is an independent copy.  The scalar expectation is absent because the Wick contraction has already been placed in the deterministic diagonal.  Applying the decoupled estimate above proves the Wick-square statement.
\end{proof}

\begin{remark}\label{rem:tensor-sharpness}
The factor \(p+\log\Lfin\) is absorbed by the strict dyadic losses below.
\end{remark}

\paragraph{Incidence counts.}\label{subsec:incidence-counts}
Apply Proposition~\ref{prop:tensor} to each pattern in
\eqref{eq:centered-normal-form}, with
\[
        a=(\underline\ell,\mu),
        \qquad b=(\underline r,\lambda),
        \qquad c=q,
        \qquad e=n.
\]
For fixed \(\nu\), set
\(\ell=\epsilon_{\nu,1}\underline\ell\) and
\(r=\epsilon_{\nu,2}\underline r\).  The sign changes are bijections on the
high shells, so the Schur bookkeeping is the same as for
\[
        n=q+\ell+r,
        \qquad |\ell|\sim |r|\sim N,
        \qquad |q|\sim Q,
        \qquad |n|\sim M,
        \qquad Q\le c_{\rm ap}N.
\]
The bounded auxiliary indices and the finite pattern set are suppressed.

\begin{lemma}[Incidence flattening bounds]\label{lem:incidence-flattenings}
For the tensor \(H\) in Lemma~\ref{lem:local-normal-form}, the four
oriented flattenings in \eqref{eq:four-flattenings} obey
\begin{equation}\label{eq:R-H-count}
        \mathfrak R_4(H)
        \lesssim N^{d/2-\Gamma_{i;j,k}}\bigl(M^{d/2}+Q^{d/2}\bigr),
\end{equation}
with constants depending only on the fixed Littlewood--Paley overlap and on the finite local Gaussian dimension.
\end{lemma}

\begin{proof}
For a flattening \(X\to Y\), call the row degree the maximal number of admissible \(X\)-tuples for a fixed \(Y\)-tuple, and the column degree the maximal number of admissible \(Y\)-tuples for a fixed \(X\)-tuple.  Since each nonzero entry is bounded by \(N^{-\Gamma_{i;j,k}}\), Schur's test bounds the flattening by \(N^{-\Gamma_{i;j,k}}\) times the square root of the product of the row and column degrees.  The bounded ranges of \(\mu\) and \(\lambda\) are absorbed in the constants.

The relevant degrees are:
\begin{center}
\begin{tabular}{lll}
\toprule
Flattening & row degree & column degree\\
\midrule
\((\ell,r,q)\to n\) & \(O(N^dQ^d)\) & \(O(1)\)\\
\((\ell,r,n)\to q\) & \(O(N^dM^d)\) & \(O(1)\)\\
\((\ell,q)\to(r,n)\) & \(O(Q^d)\) & \(O(M^d)\)\\
\((r,q)\to(\ell,n)\) & \(O(Q^d)\) & \(O(M^d)\)\\
\bottomrule
\end{tabular}
\end{center}

For \((\ell,r,q)\to n\), fixing \(n\) leaves \(q\) in a shell of
size \(O(Q^d)\) and \(\ell\) in a high shell of size \(O(N^d)\); then
\(r=n-q-\ell\) is determined.  The second row is identical with \(M\) in
place of \(Q\).

For \((\ell,q)\to(r,n)\), fixing \((r,n)\) restricts \(\ell\) to a
translate of the \(Q\)-shell, whereas fixing \((\ell,q)\) leaves at most
\(O(M^d)\) choices of \(r\).  Interchanging \(\ell\) and \(r\) gives the
last row.

Multiplying the square roots of the row and column degrees gives the profiles
\[
        N^{d/2}Q^{d/2},\qquad
        N^{d/2}M^{d/2},\qquad
        M^{d/2}Q^{d/2}.
\]
Since \(M,Q\lesssim N\), all of them are bounded by \(N^{d/2}(M^{d/2}+Q^{d/2})\).  Combining this with \eqref{eq:H-size} proves \eqref{eq:R-H-count}.
\end{proof}

After absorbing logarithms, Proposition~\ref{prop:tensor} gives the fixed-time random matrix estimate
\begin{equation}\label{eq:fixed-time-B}
        \left\|\|B_{\Lambda,N,Q,M}^{i;j,k}(t,s)\|_{\ell^2_q\to\ell^2_n}\right\|_{L^p(\Omega)}
        \lesssim_{p,T,\eps}
        N^{d/2-\Gamma_{i;j,k}+\eps}\bigl(M^{d/2+\eps}+Q^{d/2+\eps}\bigr).
\end{equation}

\paragraph{Time lift and cutoff tails.}
The fixed-time matrix estimate is upgraded to the two-time supremum by the
following Banach-valued Kolmogorov estimate.

\begin{lemma}[Banach-valued time lift]\label{lem:time-lift}
Let \(X(t,s)\), \(0\le s\le t\le T\), be a random field with values in a finite-dimensional Banach space \(B\).  Assume that for some \(A_N\), all \(p_0\ge2\), some \(\eta\in(0,1]\), and some \(\chi\ge0\),
\begin{align}
        \|\|X(t,s)\|_B\|_{L^{p_0}(\Omega)}&\le C_{p_0}A_N,\label{eq:time-lift-hyp-0}\\
        \|\|X(t,s)-X(t',s')\|_B\|_{L^{p_0}(\Omega)}
        &\le C_{p_0}A_NN^{\chi\eta}
        (|t-t'|^\eta+|s-s'|^\eta),\label{eq:time-lift-hyp-1}
\end{align}
whenever both pairs \((t,s)\) and \((t',s')\) lie in the Volterra simplex.  Then, for every fixed \(p\ge2\) and every \(\eps>0\), if \(\chi\eta<\eps/4\),
\begin{equation}\label{eq:abstract-time-lift}
        \|\|X\|_{C_{t,s}B}\|_{L^p(\Omega)}
        \le C_{p,T,\eps}A_NN^\eps .
\end{equation}
The same statement holds for a separable Banach-valued modification.  In the application below \(B=\cL(\ell_q^2,\ell_n^2)\), a finite-dimensional matrix space; the continuous-time norm is obtained from rational time skeletons.
\end{lemma}

\begin{proof}
We prove the estimate on dyadic rational grids and then pass to the separable modification.  Let
\[
        \Delta_T=\{(t,s):0\le s\le t\le T\},
        \qquad
        \Gamma_L=2^{-L}\Z\cap[0,T],
        \qquad
        \Delta_L=\Delta_T\cap\Gamma_L^2 .
\]
The cardinality of \(\Delta_L\) is \(O_T(2^{2L})\).  Choose
\[
        p_0>\max\{p,4/\eta\}+4/\eta .
\]
This leaves room for the two-dimensional grid count.

The grid at level zero is finite, so \eqref{eq:time-lift-hyp-0} gives
\[
        \left\|\sup_{(t,s)\in\Delta_0}\|X(t,s)\|_B\right\|_{L^p}
        \lesssim_{p,T} A_N .
\]
For \(L\ge0\), choose for each \((t,s)\in\Delta_{L+1}\) a point \(\pi_L(t,s)=(\bar t,\bar s)\in\Delta_L\) with
\[
        |t-\bar t|+|s-\bar s|\lesssim_T 2^{-L}.
\]
If a nearest grid point lies outside the simplex, first project to the boundary \(s=t\) and then to the grid; the same distance bound holds.  By \eqref{eq:time-lift-hyp-1},
\[
        \|X(t,s)-X(\pi_L(t,s))\|_{L^{p_0}(\Omega;B)}
        \lesssim_{p_0,T} A_NN^{\chi\eta}2^{-L\eta}.
\]
Since \(\Delta_{L+1}\) has \(O_T(2^{2L})\) points,
\[
\begin{aligned}
        &\left\|\sup_{(t,s)\in\Delta_{L+1}}
        \|X(t,s)-X(\pi_L(t,s))\|_B\right\|_{L^p} \\
        &\qquad\le
        \left\|\sup_{(t,s)\in\Delta_{L+1}}
        \|X(t,s)-X(\pi_L(t,s))\|_B\right\|_{L^{p_0}} \\
        &\qquad\lesssim
        2^{2L/p_0} A_NN^{\chi\eta}2^{-L\eta}.
\end{aligned}
\]
The series \(\sum_L2^{-L(\eta-2/p_0)}\) converges.  The dyadic chaining estimate gives
\[
        \left\|\sup_{(t,s)\in\bigcup_L\Delta_L}\|X(t,s)\|_B\right\|_{L^p}
        \lesssim_{p,T,\eta} A_N(1+N^{\chi\eta}).
\]
For a separable continuous modification, the supremum over the rational skeleton equals the \(C(\Delta_T;B)\)-norm.  In the finite matrix application this modification is obtained from the increment estimates for the matrix entries and the finite-dimensional operator norm.  Finally \(\chi\eta<\eps/4\) gives \(N^{\chi\eta}\le N^\eps\) for \(N\ge1\), and finitely many low dyadic shells are absorbed into the constant.  This proves \eqref{eq:abstract-time-lift}.
\end{proof}

We apply Lemma~\ref{lem:time-lift} with \(B=\cL(\ell_q^2,\ell_n^2)\).  Fix the strict dyadic loss \(\eps\).  Choose \(0<\theta\le\eta_0\) so that \(\chi_{i;j,k}\theta<\eps/4\), and set \(\eta=\theta\).  The fixed-time estimate is \eqref{eq:fixed-time-B}.  It remains to verify the increment hypothesis in \eqref{eq:time-lift-hyp-1}.

A centered summand has the form
\[
        K_{t,s}\wick{X_sY_t},
        \qquad
        K_{t,s}=K_i(t-s,q+\ell),
        \quad
        X_s=\widehat\Psi_j(\ell,s),
        \quad
        Y_t=\widehat\Psi_k(r,t).
\]
The degree-two Wick projection is bilinear after the scalar expectation has been removed.  Hence
\[
        \wick{X_sY_t}-\wick{X_sY_{t'}}
        =\wick{X_s(Y_t-Y_{t'})},
        \qquad
        \wick{X_sY_t}-\wick{X_{s'}Y_t}
        =\wick{(X_s-X_{s'})Y_t} .
\]
The upper-time increment is
\[
\begin{aligned}
        K_{t,s}\wick{X_sY_t}-K_{t',s}\wick{X_sY_{t'}}
        &=(K_{t,s}-K_{t',s})\wick{X_sY_t} \\
        &\quad+K_{t',s}\wick{X_s(Y_t-Y_{t'})},
\end{aligned}
\]
and the lower-time increment is
\[
\begin{aligned}
        K_{t,s}\wick{X_sY_t}-K_{t,s'}\wick{X_{s'}Y_t}
        &=(K_{t,s}-K_{t,s'})\wick{X_sY_t} \\
        &\quad+K_{t,s'}\wick{(X_s-X_{s'})Y_t} .
\end{aligned}
\]
For two arbitrary Volterra pairs \((t,s),(t',s')\in\Delta_T:=\{0\le s\le t\le T\}\), the telescope is taken along a path that stays inside the simplex.  Put \(u=\min\{s,s'\}\) and write, for \(F(t,s):=K_{t,s}\wick{X_sY_t}\),
\[
        F(t,s)-F(t',s')=\bigl(F(t,s)-F(t,u)\bigr)
        +\bigl(F(t,u)-F(t',u)\bigr)
        +\bigl(F(t',u)-F(t',s')\bigr).
\]
All four points \((t,s),(t,u),(t',u),(t',s')\) belong to \(\Delta_T\), since \(u\le s\le t\) and \(u\le s'\le t'\).  The first and third terms are lower-time increments, while the middle term is an upper-time increment at the common lower time \(u\).  Their increment lengths are bounded by \(|s-s'|\), \(|t-t'|\) and \(|s-s'|\), respectively.  Thus each increment tensor is a finite sum of tensors with one increment factor.  In same-color blocks the scalar covariance pairing has already been subtracted at finite cutoff; the increment tensors are centered second chaoses and are decoupled as in Lemma~\ref{lem:local-normal-form}.  The spatial incidence remains \(n=q+\ell+r\) in every telescope summand.

Assumption~\ref{ass:coeff} gives the increment factor \(|t-t'|^\theta N^{\chi_{i;j,k}\theta}\) or \(|s-s'|^\theta N^{\chi_{i;j,k}\theta}\), while the deterministic flattening counts are unchanged.  Applying Proposition~\ref{prop:tensor} to these finitely many increment tensors and then Lemma~\ref{lem:time-lift} gives
\begin{equation}\label{eq:time-lift}
        \left\|\|B_{\Lambda,N,Q,M}^{i;j,k}\|_{C_{t,s}\cL(\ell_q^2,\ell_n^2)}\right\|_{L^p(\Omega)}
        \lesssim_{p,T,\eps}
        N^{d/2-\Gamma_{i;j,k}+\eps}\bigl(M^{d/2+\eps}+Q^{d/2+\eps}\bigr).
\end{equation}
For cutoff tails, fix \((N,Q,M)\).  With sharp Galerkin cutoffs the block is eventually constant once \(\Lambda\gg N+Q+M\).  With a smooth stochastic-leg cutoff the block contains only finitely many Fourier modes and the cutoff symbols converge entrywise; hence the finite matrix converges in \(C_{t,s}\cL(\ell_q^2,\ell_n^2)\).  The same statement holds for the increment tensors.  Global cutoff Cauchy convergence is obtained by combining this fixed-block convergence with the summable dyadic majorants in Section~\ref{sec:assembly}.

\subsection{Dyadic summation and cutoff convergence}\label{sec:assembly}

We use the following elementary dyadic estimates.  For any \(a\in\R\), dyadic \(N\ge1\), and \(\rho>0\),
\begin{equation}\label{eq:dyadic-sum-basic}
        \sum_{L\lesssim N}L^a\lesssim_\rho N^{a_++\rho},
        \qquad
        \sum_{L\gtrsim N}L^{-a}\lesssim_\rho N^{-a+\rho}\quad(a>0),
\end{equation}
where \(a_+=\max\{a,0\}\).  The inhomogeneous low block is absorbed in the implicit constant.  Every occurrence of \(\rho\) below is taken smaller than the fixed strict loss in Definition~\ref{def:admissible}.

We now insert the deterministic input and sum \eqref{eq:time-lift}.  Fix a triple \((i;j,k)\), and write
\[
        \lambda=\lambda_i,
        \qquad
        \Gamma=\Gamma_{i;j,k}.
\]
For deterministic \(w\), the transition from the matrix estimate to the inserted operator is made before the dyadic summation.  For each \(t\), the Volterra expression is bounded by the \(C_{t,s}\cL(\ell_q^2,\ell_n^2)\)-norm of the kernel times \(\|P_Qw\|_{L_T^\infty L^2}\); Minkowski's inequality in \(s\) and the estimate \(T\le T_0\) give the same bound in \(C_TL^2\) and in \(L_T^1L^2\).  Hence
\begin{equation}\label{eq:insert-L2}
        \|P_MB_{N,Q,M}^{i;j,k}P_Qw\|_{L^p(\Omega;C_TL^2)}
        \lesssim
        N^{d/2-\Gamma+\eps}\bigl(M^{d/2+\eps}+Q^{d/2+\eps}\bigr)
        \|P_Qw\|_{L_T^\infty L^2}.
\end{equation}
Using \(\|P_Qw\|_{L_T^\infty L^2}\le Q^{-\sigma}\|w\|_{L_T^\infty B^\sigma_{2,\infty}}\), the Sobolev block is bounded by
\begin{equation}\label{eq:H-block}
        N^{d/2-\Gamma+\eps}M^{s-\lambda}
        \bigl(M^{d/2+\eps}+Q^{d/2+\eps}\bigr)Q^{-\sigma}
        \|w\|_{X_T^\sigma}.
\end{equation}
For the Besov target,
\begin{equation}\label{eq:B-block}
        M^{\sigma-\lambda}
        \|P_MB_{N,Q,M}^{i;j,k}P_Qw\|_{L^p(\Omega;L_T^1L^2)}
        \lesssim_T
        N^{d/2-\Gamma+\eps}M^{\sigma-\lambda}
        \bigl(M^{d/2+\eps}+Q^{d/2+\eps}\bigr)Q^{-\sigma}
        \|w\|_{X_T^\sigma}.
\end{equation}

\begin{lemma}[Dyadic finite-set/tail assembly]\label{lem:dyadic-tail-assembly}
Let \(\mathfrak D\) be the countable set of dyadic triples \(\mathbf d=(N,Q,M)\) satisfying \(Q\le c_{\rm ap}N\) and \(M\lesssim N\).  Let \(X\) and \(Y\) be Banach spaces, and let \(U_{\Lambda,\mathbf d}:X\to Y\) be finite-rank operators.  Assume the following two conditions.
\begin{enumerate}[label=(\roman*)]
\item For each fixed \(\mathbf d\), the family \(U_{\Lambda,\mathbf d}\) is Cauchy in \(\mathcal L(X,Y)\) as \(\Lambda\to\infty\).
\item There exist deterministic numbers \(a_{\mathbf d}\ge0\) with \(\sum_{\mathbf d\in\mathfrak D}a_{\mathbf d}<\infty\) such that
\[
        \sup_{\Lambda,\Lambda'}
        \|U_{\Lambda,\mathbf d}-U_{\Lambda',\mathbf d}\|_{\mathcal L(X,Y)}
        \le a_{\mathbf d}
        \qquad \text{for every }\mathbf d\in\mathfrak D .
\]
\end{enumerate}
Then \(\sum_{\mathbf d\in\mathfrak D}U_{\Lambda,\mathbf d}\) is Cauchy in \(\mathcal L(X,Y)\), and the limit is independent of the order in which finite dyadic sets exhaust \(\mathfrak D\).

For a \(B_{2,\infty}^{\zeta}\)-target, the same conclusion follows from output-shell majorants.  Suppose that the contribution of all triples with output shell \(M\) is bounded by \(b_M\ge0\), with \(\sum_Mb_M<\infty\).  Then, for \(\|x\|_X\le1\),
\begin{equation}\label{eq:Besov-tail-assembly}
        \sup_{M\ge M_0} M^\zeta
        \left\|P_M\sum_{(N,Q):(N,Q,M)\in\mathfrak D} U_{\Lambda,N,Q,M}x\right\|_{L_T^1L^2}
        \le \sum_{M\ge M_0} b_M
        \longrightarrow0 .
\end{equation}
The same implications hold in \(L^p(\Omega)\) if \(a_{\mathbf d}\) and \(b_M\) dominate the corresponding \(L^p(\Omega)\)-norms of the block operator differences.  If there are random majorants \(A_{\mathbf d}(\omega)\) with \(\sum_{\mathbf d}A_{\mathbf d}(\omega)<\infty\) almost surely and
\[
        \sup_{\Lambda,\Lambda'}
        \|U_{\Lambda,\mathbf d}(\omega)-U_{\Lambda',\mathbf d}(\omega)\|_{\mathcal L(X,Y)}
        \le A_{\mathbf d}(\omega),
\]
then the same finite-set/tail argument gives almost sure operator-norm Cauchy convergence.  In particular, summable \(L^1(\Omega)\) majorants imply this pathwise hypothesis by Fubini.
\end{lemma}

\begin{proof}
Let \(\delta>0\).  Choose a finite set \(F\subset\mathfrak D\) such that \(\sum_{\mathbf d\notin F}a_{\mathbf d}<\delta\).  The tail satisfies
\[
        \left\|\sum_{\mathbf d\notin F}
        (U_{\Lambda,\mathbf d}-U_{\Lambda',\mathbf d})\right\|_{\mathcal L(X,Y)}
        \le \sum_{\mathbf d\notin F}a_{\mathbf d}<\delta,
\]
uniformly in \(\Lambda,\Lambda'\).  On the finite set \(F\), condition (i) gives Cauchy convergence of the finite sum.  This proves the Cauchy property.  Since the proof uses only tails of an \(\ell^1\)-family, the limit is unconditional with respect to the chosen finite exhaustion.

For the Besov target, the norm is the supremum over output shells after multiplication by the weight \(M^\zeta\).  Estimate \eqref{eq:Besov-tail-assembly} is just the inequality \(\sup_{M\ge M_0}x_M\le\sum_{M\ge M_0}x_M\) for non-negative shell bounds.  Thus an \(\ell^1_M\) output tail gives norm Cauchy convergence in \(B_{2,\infty}^\zeta\).  A uniform shell bound \(\sup_M b_M<\infty\) would give boundedness of each output shell but would not force the high-output tail to vanish.

The \(L^p\)-version is the same deterministic proof applied to \(L^p(\Omega)\)-norms, using Minkowski's inequality for finite sums and the \(\ell^1\) tail.  The almost sure version is pathwise.  If \(\sum_{\mathbf d}\E A_{\mathbf d}<\infty\), then Fubini gives \(\sum_{\mathbf d}A_{\mathbf d}(\omega)<\infty\) for almost every \(\omega\); applying the deterministic finite-set/tail proof on that event proves the assertion.
\end{proof}

\paragraph{Sobolev target.}
Write
\[
        \lambda=\lambda_i,\qquad \Gamma=\Gamma_{i;j,k},\qquad \beta=s-\lambda .
\]
The centered window gives
\begin{equation}\label{eq:window-short}
        \Gamma>\frac d2+10\eps,
        \qquad
        s<\lambda+\Gamma-d-10\eps,
        \qquad
        \max\{0,d-\Gamma\}+10\eps<\sigma<\lambda+\Gamma-d-10\eps .
\end{equation}
The first inequality gives the summability of the bare high-loop factor \(N^{d/2-\Gamma}\), with room for the logarithmic and time-lift losses.

Consider first the \(M^{d/2}\)-part of \eqref{eq:H-block}.  Since \(\sigma>0\), the dyadic sum in \(Q\le c_{\rm ap}N\) contributes at most a logarithmic factor, absorbed by \(N^\eps\).  The sum in \(M\lesssim N\) is
\[
        \sum_{M\lesssim N}M^{s-\lambda+d/2+\eps}
        \lesssim
        \begin{cases}
        1, & s-\lambda+d/2<-2\eps,\\
        N^{s-\lambda+d/2+2\eps}, & s-\lambda+d/2\ge -2\eps.
        \end{cases}
\]
If the first case occurs, the remaining high-shell power is \(N^{d/2-\Gamma+O(\eps)}\), which is summable because \(\Gamma>d/2\).  In the second case the high-shell power is
\[
        N^{d/2-\Gamma}N^{s-\lambda+d/2+O(\eps)}
        =N^{s-\lambda+d-\Gamma+O(\eps)},
\]
which is summable by the upper bound \(s<\lambda+\Gamma-d-10\eps\).  Thus the \(M^{d/2}\)-branch maps into \(C_TH^{s-\lambda}\).

For the \(Q^{d/2}\)-part of \eqref{eq:H-block}, we use the positive-part form of the same calculation:
\begin{equation}\label{eq:H-Q-positive}
        \sum_{M\lesssim N}M^{s-\lambda+\eps}
        \lesssim N^{(s-\lambda)_++2\eps},
        \qquad
        \sum_{Q\le c_{\rm ap}N}Q^{d/2-\sigma+\eps}
        \lesssim N^{(d/2-\sigma)_++2\eps}.
\end{equation}
The high-shell exponent is therefore
\begin{equation}\label{eq:H-Q-exponent}
        d/2-\Gamma+(s-\lambda)_++(d/2-\sigma)_+ +O(\eps).
\end{equation}
We check that it is negative under \eqref{eq:window-short}.  If \(\Gamma\le d\), then \(s<\lambda+\Gamma-d\le\lambda\), hence \((s-\lambda)_+=0\), and \(\sigma>d-\Gamma\).  When \(\sigma<d/2\), this gives
\[
        d/2-\Gamma+d/2-\sigma=d-\Gamma-\sigma<0;
\]
when \(\sigma\ge d/2\), the exponent is at most \(d/2-\Gamma<0\).  If \(\Gamma>d\), then \((s-\lambda)_+<\Gamma-d\) and \(\sigma>0\).  If \(\sigma<d/2\), the exponent is strictly smaller than
\[
        d/2-\Gamma+(\Gamma-d)+d/2-\sigma=-\sigma<0;
\]
if \(\sigma\ge d/2\), it is smaller than \(-d/2\).  The strict losses absorb the logarithmic factors.  This proves the centered bound into \(C_TH^{s-\lambda_i}\).

\paragraph{Besov target.}
For each output scale \(M\), let \(C_M\) denote the deterministic majorant obtained from \eqref{eq:B-block} after summing \(N\gtrsim M\) and \(Q\le c_{\rm ap}N\) with the Besov weight \(M^{\sigma-\lambda}\).  The estimates below show that the strict window implies \(\sum_M C_M<\infty\).  Consequently,
\[
        \left\|\sup_M X_M\right\|_{L^p(\Omega)}
        \le \left\|\sum_M X_M\right\|_{L^p(\Omega)}
        \le \sum_M \|X_M\|_{L^p(\Omega)}
\]
for nonnegative shell norms \(X_M\).

For the \(M^{d/2}\)-part, the \(Q\)-sum contributes only a logarithmic or bounded factor because \(\sigma>0\).  Since \(\Gamma>d/2\),
\begin{equation}\label{eq:B-M-general}
        M^{\sigma-\lambda+d/2+\eps}
        \sum_{N\gtrsim M}N^{d/2-\Gamma+\eps}
        \lesssim
        M^{\sigma-\lambda+d-\Gamma+O(\eps)}.
\end{equation}
The exponent is negative by \(\sigma<\lambda+\Gamma-d-10\eps\), so this branch is uniformly bounded and has a high-output tail.

For the \(Q^{d/2}\)-part, there are two cases.  If \(\sigma<d/2\), then
\[
        \sum_{Q\le c_{\rm ap}N}Q^{d/2-\sigma+\eps}
        \lesssim N^{d/2-\sigma+2\eps}.
\]
The \(N\)-tail is summable because \(\sigma>d-\Gamma+10\eps\) in this case, and it gives
\begin{equation}\label{eq:B-Q-low-sigma}
        M^{\sigma-\lambda}
        \sum_{N\gtrsim M}N^{d-\Gamma-\sigma+O(\eps)}
        \lesssim
        M^{d-\Gamma-\lambda+O(\eps)}.
\end{equation}
The last exponent is negative because the upper end of the window is positive, hence \(\lambda+\Gamma>d\).  If \(\sigma\ge d/2\), then the inner \(Q\)-sum costs only a logarithm and
\begin{equation}\label{eq:B-Q-high-sigma}
        M^{\sigma-\lambda}
        \sum_{N\gtrsim M}N^{d/2-\Gamma+O(\eps)}
        \lesssim
        M^{\sigma-\lambda+d/2-\Gamma+O(\eps)}.
\end{equation}
This is even better than required, because \(\sigma<\lambda+\Gamma-d\) implies
\(
        \sigma-\lambda+d/2-\Gamma<-d/2.
\)
The estimates \eqref{eq:B-M-general}, \eqref{eq:B-Q-low-sigma} and \eqref{eq:B-Q-high-sigma} are negative powers of \(M\) after the strict losses are fixed.  Hence \(\sum_M C_M<\infty\), and the displayed \(\ell^1_M\) domination gives both the \(L^p\)-bound of the Besov supremum and the high-output tail needed for cutoff convergence.  Thus \eqref{eq:centered-B} follows.

\paragraph{Cutoff convergence and completion of the proof.}
Let \(\mathfrak D\) denote the countable set of dyadic triples \((N,Q,M)\) satisfying \(Q\le c_{\rm ap}N\) and \(M\lesssim N\).  For each \((N,Q,M)\in\mathfrak D\), the Galerkin cutoff kernel is eventually constant as \(\Lambda\to\infty\).  We denote this finite-block limit by \(B_{N,Q,M}^{i;j,k}\).

The estimates above give convergence of the dyadic series in \(L^p(\Omega)\).  For every finite subset \(F\subset\mathfrak D\), the finite sum over \(F\) is eventually independent of \(\Lambda\).  On the complement \(\mathfrak D\setminus F\), the Sobolev majorants are \(\ell^1_{N,Q,M}\)-summable and the Besov shell majorants are \(\ell^1_M\)-summable by the strict inequalities in \eqref{eq:window-short}.  Hence the tails are uniformly small in
\[
        L^p\bigl(\Omega;\cL(X_T^\sigma,C_TH^{s-\lambda_i})\bigr)
        \cap
        L^p\bigl(\Omega;\cL(X_T^\sigma,L_T^1B^{\sigma-\lambda_i}_{2,\infty})\bigr),
\]
after \(F\) is chosen large.  This proves \(L^p\)-convergence of \(B_\Lambda^{i;j,k}\) to the dyadic series \(\sum_{\mathfrak D}B_{N,Q,M}^{i;j,k}\).

We now prove almost-sure convergence from the same strict dyadic
summability.  Choose support constants \(0<c_0<C_0<\infty\) for \(\rho_N\)
and set
\[
        \mathscr L_N
        :=\{\Lambda\in2^{\N_0}:c_0N\le\Lambda\le C_0N\}.
\]
For each dyadic triple, let
\[
\mathscr C_{N,Q,M}:=
\{B_{N,Q,M}\}
\cup
\{B_{\Lambda,N,Q,M},
  B_{\Lambda,N,Q,M}-B_{N,Q,M}:\Lambda\in\mathscr L_N\}.
\]
The cardinality of this family is uniformly bounded.  Outside
\(\mathscr L_N\), a sharp-cutoff block is either zero or already equal to its
limit, and the tensor and time-lift estimates hold uniformly for every
member of \(\mathscr C_{N,Q,M}\).

Define the weighted random block norms
\[
\begin{aligned}
\mathcal A^H_{N,Q,M}
&:=M^{s-\lambda_i}Q^{-\sigma}
  \sup_{C\in\mathscr C_{N,Q,M}}
  \|C\|_{C_{t,s}\cL(\ell_q^2,\ell_n^2)},\\
\mathcal A^B_{N,Q,M}
&:=M^{\sigma-\lambda_i}Q^{-\sigma}
  \sup_{C\in\mathscr C_{N,Q,M}}
  \|C\|_{C_{t,s}\cL(\ell_q^2,\ell_n^2)}.
\end{aligned}
\]
The Sobolev calculation above gives
\(\sum_{N,Q,M}\E\mathcal A^H_{N,Q,M}<\infty\), while the
\(\ell^1_M\) Besov calculation gives
\(\sum_{N,Q,M}\E\mathcal A^B_{N,Q,M}<\infty\).  Here we use the
\(p=2\) block estimate and the uniformly bounded cardinality of
\(\mathscr C_{N,Q,M}\).  Fubini's theorem therefore yields, on one
full-probability event,
\[
        \sum_{N,Q,M}
        \bigl(\mathcal A^H_{N,Q,M}+\mathcal A^B_{N,Q,M}\bigr)<\infty
\]
for every label in the finite family.  The deterministic finite-set/tail
argument in Lemma~\ref{lem:dyadic-tail-assembly} now applies pathwise.  As
\(\Lambda\to\infty\), every changing block belongs to the high-frequency
tail of the above summable family, and hence
\(B_\Lambda^{i;j,k}\to B^{i;j,k}\) almost surely in
\eqref{eq:centered-op-topology}.  The pathwise limit agrees with the
\(L^p\)-limit.

Measurability follows by taking rational time pairs and rational unit
vectors in each finite matrix block.  Lemma~\ref{lem:time-lift} supplies a
separable continuous modification, and the summable dyadic series defines
the resulting bounded operator on \(X_T^\sigma\).

For each label \(\tau\in\cM\) with a nonzero retained diagonal, Definition~\ref{def:contraction-resolution} gives fixed-block convergence of \(\Ret^\tau_{\Lambda,N,Q}\) and the dyadic bound \eqref{eq:retained-block-bound}.  Summing \eqref{eq:retained-block-bound} with \(\delta_\tau\) in place of \(\Delta\), exactly as in the proof of Proposition~\ref{prop:diagonal}, gives \eqref{eq:retained-main-bound}.  The convergence of \(\Ret^\tau_\Lambda\) follows by fixed-block convergence on a finite dyadic set and by the summable tail bound on the complement.  If \(\Ret^\tau_\Lambda=0\), set \(\Ret^\tau=0\).

The finite identity \eqref{eq:finite-ren-equals-B} gives
\[
        T_{\Lambda}^{\tau,\mathfrak P}=\Ret^\tau_\Lambda+B^\tau_\Lambda .
\]
Combining deterministic convergence of \(\Ret^\tau_\Lambda\) with the centered convergence of \(B^\tau_\Lambda\) proves the assembled convergence \eqref{eq:assembled-limit}.  The special raw and fully Wick-subtracted choices are the two cases recorded in Remark~\ref{rem:raw-wick}.  This proves Theorem~\ref{thm:main}.

\section{A distinct-speed diagonal estimate}\label{sec:verification}

We verify the retained-diagonal criterion for a distinct-speed wave or
Klein--Gordon contraction.

\begin{proposition}[Distinct-speed wave/Klein--Gordon diagonals]\label{prop:wave-verification}
Let \(\tau=(i;j,j)\) and work in the standard high-frequency wave/Klein--Gordon normalization: the outer Duhamel factor has gain \(\lambda_i=1\), the stochastic convolution leg has gain \(\alpha_j=1\), and hence \(\Gamma_{i;j,j}=3\) in Assumption~\ref{ass:coeff}.  Assume that, at high frequency,
\[
        \omega_a(n)=(m_a^2+c_a^2|n|^2)^{1/2},
        \qquad m_a>0,
        \qquad c_i\ne c_j .
\]
Suppose that
\[
        K_i(t,n)=\frac{\sin(t\omega_i(n))}{\omega_i(n)}
\]
and that \(\Psi_j\) is the corresponding stochastic convolution.  Assume
\(d<4\) and choose \(\eps>0\) so that \(20\eps<4-d\).  Then the same-color
diagonal kernel satisfies Assumption~\ref{ass:diag} on a sufficiently small
proportional aperture \(|q|\le c_{\rm ap}N\), with phase gain \(\vartheta=1\)
and any retained gain \(\Delta\) satisfying
\[
        20\eps<\Delta<4-d.
\]
More precisely, after the trigonometric expansion the principal phases are
\begin{equation}\label{eq:wave-phases}
        \Phi_{\epsilon_1,\epsilon_2}(q,\ell)
        =\epsilon_1\omega_i(q+\ell)+\epsilon_2\omega_j(\ell),
        \qquad \epsilon_1,\epsilon_2\in\{+1,-1\},
\end{equation}
and
\begin{equation}\label{eq:wave-phase-lower}
        |\Phi_{\epsilon_1,\epsilon_2}(q,\ell)|\gtrsim N,
        \qquad |\ell|\sim N,
        \qquad |q|\le c_{\rm ap}N .
\end{equation}
Consequently, in this normalization, a retained same-color diagonal has loop gain
\begin{equation}\label{eq:wave-loop-gain}
        0<\Delta<1+\Gamma_{i;j,j}-d=4-d.
\end{equation}
In dimension three, every fixed \(0<\Delta<1\) is available after choosing
\(\eps\) sufficiently small.
\end{proposition}

\begin{proof}
The covariance of the wave/Klein--Gordon stochastic convolution is, for \(\omega_j=\omega_j(\ell)\),
\begin{equation}\label{eq:wave-cov-identity}
\begin{aligned}
        \sigma_{jj}(\ell;s,t)
        &=\int_0^s\frac{\sin((s-r)\omega_j)}{\omega_j}
        \frac{\sin((t-r)\omega_j)}{\omega_j}\, \dd r  \\
        &=\frac{s}{2\omega_j^2}\cos((t-s)\omega_j)
        +\frac{\sin((t-s)\omega_j)-\sin((t+s)\omega_j)}{4\omega_j^3}.
\end{aligned}
\end{equation}
The two sine terms in \eqref{eq:wave-cov-identity} are absolutely bounded by \(\lesssim N^{-3}\), uniformly for \(0\le s\le t\le T_0\).  After multiplication by the Duhamel factor \(\omega_i(q+\ell)^{-1}\sim N^{-1}\), their high-loop contribution is \(\lesssim N^{d-4}\), which is included in the remainder for every gain below \(4-d\).

For the principal term, expanding \(\sin((t-s)\omega_i(q+\ell))\cos((t-s)\omega_j(\ell))\) into exponentials gives the phases \eqref{eq:wave-phases}.  The amplitude has the form
\[
        a(q,\ell;t,s)=c_0 s\,\omega_i(q+\ell)^{-1}\omega_j(\ell)^{-2}
\]
up to constants and signs independent of the dyadic scales.  Hence
\[
        |a(q,\ell;t,s)|+|\partial_s a(q,\ell;t,s)|
        \lesssim_T N^{-3}=N^{-\Gamma_{i;j,j}}
\]
under the standard wave-gain normalization \(\lambda_i=\alpha_j=1\), which is the amplitude bound required in \eqref{eq:amp-bound}.  It remains to prove the phase lower bound.  The sum phases satisfy \(|\omega_i(q+\ell)+\omega_j(\ell)|\gtrsim N\).  For the difference phase, using \(\omega_a(n)=c_a|n|+O(|n|^{-1})\) and the mean value theorem,
\[
        |\omega_i(q+\ell)-\omega_j(\ell)|
        \ge |c_i-c_j|\,|\ell|-C(1+c_i)|q|-CN^{-1}.
\]
Choosing \(c_{\rm ap}\) sufficiently small in terms of \(c_i,c_j,m_i,m_j\) gives \eqref{eq:wave-phase-lower} for all large \(N\).  The finitely many low dyadic shells are not treated by a phase denominator; their contribution is absorbed into the finite low-frequency constant in the deterministic diagonal estimate.  Thus \(\vartheta=1\), and the loop gain condition in Assumption~\ref{ass:diag} becomes \(1+\Gamma_{i;j,j}-d\ge\Delta\).
\end{proof}

\begin{corollary}[Three-dimensional equal-gain wave/Klein--Gordon block]\label{cor:three-d-wave}
Take \(d=3\) and \(\lambda_i=\alpha_j=\alpha_k=1\).  Then \(\Gamma_{i;j,k}=3\).  The deterministic resonant threshold \(\Gamma>d\) is missed at equality, while the centered window of Definition~\ref{def:admissible} allows, up to the strict losses,
\[
        s<1,
        \qquad
        0<\sigma<1 .
\]
For the distinct-speed same-color contraction in Proposition~\ref{prop:wave-verification}, every fixed retained deterministic gain \(0<\Delta<1\) is available after choosing the losses.  Hence the retained diagonal is admissible for every target exponent \(s<1\).  For this equal-gain branch the centered estimate gives the random operator below the deterministic product threshold, and the same-color diagonal is reduced to the scalar Volterra estimate with gain \(0<\Delta<1\).
\end{corollary}

\appendix

\section{Frequency-envelope extension}\label{subsec:envelope-variant}

\subsection{Gaussian time profiles}
The notation \(\Psi_a\) may be realized by a stochastic convolution, but the arguments below use only the finite Gaussian Hilbert-space structure and the Fourier-diagonal covariance.  The common object is the isonormal representation
\[
        \wh\Psi_a(n,t)=\isoW(h_{a,n,t}).
\]
Space-time white-noise forcing and random initial data are two useful special cases.  For a stochastic convolution driven by a Brownian Fourier mode one may take, for each fixed \(n\),
\[
        h_{a,n,t}(r)=\one_{0\le r\le t}K_a(t-r,n)b_a(n)
        \quad\hbox{in }L^2([0,T_0]),
\]
so that \(\isoW(h_{a,n,t})=\int_0^tK_a(t-r,n)b_a(n)\,\dd\beta_{a,n}(r)\).  For random initial data evolved by a linear flow one may instead take
\[
        h_{a,n,t}=e^{\ii t\phi_a(n)}a_a(n)e_{a,n},
        \qquad
        \wh\Psi_a(n,t)=e^{\ii t\phi_a(n)}a_a(n)g_{a,n},
\]
or the analogous finite sum of sine and cosine profiles in wave variables.  The first example has Volterra time covariance, while the second has rank-one or finite-rank time covariance.  Both satisfy the argument below whenever the stated dyadic amplitude and increment bounds hold.

\subsection{Envelope assumptions and theorem}

The power-law assumptions in Assumption~\ref{ass:coeff} can be replaced by dyadic envelopes without changing the finite Wick split or the centered tensor estimate.  The resulting theorem allows frequency-dependent amplitudes and non-Volterra time covariance.

Let \(K_i(N)>0\) be the dyadic size of the Duhamel multiplier in channel \(i\), let \(A_a(N)>0\) be the dyadic size of the Gaussian leg of color \(a\), and let \(\Theta_i(N),\Theta_a(N)\ge1\) be the corresponding time-oscillation scales.  The envelope coefficient assumptions are
\begin{align}
        |K_i(t,n)|&\le C K_i(N), \label{eq:env-K}\\
        |K_i(t,n)-K_i(t',n)|&\le C_\theta |t-t'|^\theta K_i(N)\Theta_i(N)^\theta, \label{eq:env-K-increment}\\
        \|h_{a,n,t}\|_{\mathfrak H}&\le C A_a(N), \label{eq:env-A}\\
        \|h_{a,n,t}-h_{a,n,t'}\|_{\mathfrak H}&\le C_\theta |t-t'|^\theta A_a(N)\Theta_a(N)^\theta, \label{eq:env-A-increment}
\end{align}
for \(|n|\sim N\), \(0\le t,t'\le T_0\), and \(0<\theta\le\eta_0\).  The covariance is still Fourier diagonal as in \eqref{eq:conv-cov}.  For a mixed label \(\tau=(i;j,k)\) define
\begin{equation}\label{eq:env-GTheta}
        G_\tau(N):=K_i(N)A_j(N)A_k(N),
        \qquad
        \Theta_\tau(N):=1+\Theta_i(N)+\Theta_j(N)+\Theta_k(N).
\end{equation}
We assume the following time-lift absorption condition.  For every strict dyadic loss \(\eps>0\) used in the summation, there exists \(0<\theta\le\eta_0\) such that
\begin{equation}\label{eq:time-lift-absorb}
        \sup_{N\ge1}\Theta_\tau(N)^\theta N^{-\eps}<\infty
\end{equation}
for every label in the finite family.  In the power-law case \(K_i(N)=N^{-\lambda_i}\), \(A_a(N)=N^{-\alpha_a}\), \(\Theta_i(N)=N^{\lambda_i}\), \(\Theta_a(N)=N^{\alpha_a}\), this is exactly the small-\(\theta\) loss absorption used in Assumption~\ref{ass:coeff}.

For \(\beta,\zeta,\sigma\in\mathbb R\), define the centered envelope majorants
\begin{align}
\mathfrak S_{\tau,H}^{\eps}(\beta,\sigma)
&:=\sum_N\sum_{Q\le c_{\rm ap}N}\sum_{M\lesssim N}
        M^\beta Q^{-\sigma}
        N^{d/2+\eps}G_\tau(N)
        \bigl(M^{d/2+\eps}+Q^{d/2+\eps}\bigr), \label{eq:env-SH}\\
\mathfrak b_{\tau,M}^{\eps}(\zeta,\sigma)
&:= M^\zeta\sum_{N\gtrsim M}\sum_{Q\le c_{\rm ap}N}
        Q^{-\sigma}
        N^{d/2+\eps}G_\tau(N)
        \bigl(M^{d/2+\eps}+Q^{d/2+\eps}\bigr), \label{eq:env-bM}\\
\mathfrak S_{\tau,B}^{\eps}(\zeta,\sigma)
&:=\sum_M \mathfrak b_{\tau,M}^{\eps}(\zeta,\sigma). \label{eq:env-SB}
\end{align}
Here \(\beta\) is the Sobolev output exponent and \(\zeta\) is the Besov output exponent.  The \(\ell^1_M\) condition in \eqref{eq:env-SB} is part of the convergence hypothesis: a uniform bound \(\sup_M\mathfrak b_{\tau,M}^{\eps}<\infty\) gives boundedness of individual output shells, whereas norm convergence in \(B_{2,\infty}^\zeta\) requires a high-output tail.  The summability in \(M\) supplies this tail.  The power-law theorem is recovered by taking
\begin{equation}\label{eq:power-specialization}
        \beta=s-\lambda_i,
        \qquad
        \zeta=\sigma-\lambda_i,
        \qquad
        G_\tau(N)=N^{-\lambda_i-\alpha_j-\alpha_k}.
\end{equation}
Then the finiteness of \eqref{eq:env-SH} and \eqref{eq:env-SB} is precisely the dyadic content of the centered window in Definition~\ref{def:admissible}, up to the displayed strict losses.

With \(X_T^\sigma\) as in \eqref{eq:X-space}, set
\begin{equation}\label{eq:env-source-target}
        Y_T^{\beta,\zeta}:=C_TH^\beta\cap L_T^1B_{2,\infty}^\zeta.
\end{equation}
For retained deterministic diagonals in the envelope formulation, let \(\mathcal X_T\) be a model-dependent Banach source space equipped with dyadic localized seminorms \(\|P_Qw\|_{\mathcal X_{T,Q}}\).  Put
\[
        \mathfrak D_2:=\{(N,Q):Q\le c_{\rm ap}N\}.
\]
We assume the localization is compatible with the retained dyadic majorant in the following tail sense: for every directed family of finite subsets \(F\subset\mathfrak D_2\) increasing to \(\mathfrak D_2\),
\begin{equation}\label{eq:env-X-localized}
        \sum_{(N,Q)\in\mathfrak D_2} b_\tau(N,Q)\|P_Qw\|_{\mathcal X_{T,Q}}
        \le C_{\mathcal X}\|w\|_{\mathcal X_T},
        \qquad
        \sup_{\|w\|_{\mathcal X_T}\le1}\sum_{(N,Q)\in\mathfrak D_2\setminus F} b_\tau(N,Q)\|P_Qw\|_{\mathcal X_{T,Q}}
        \longrightarrow0 .
\end{equation}
A sufficient special case is \(\sup_Q\|P_Qw\|_{\mathcal X_{T,Q}}\le C\|w\|_{\mathcal X_T}\) together with \(\sum_{(N,Q)\in\mathfrak D_2}b_\tau(N,Q)<\infty\).  The retained branches are required to satisfy directly the target bound
\begin{equation}\label{eq:env-retained-bound}
        \|\Ret_{\Lambda,N,Q}^\tau P_Qw\|_{Y_T^{\beta,\zeta}}
        \le C_{\mathfrak P} b_\tau(N,Q)\|P_Qw\|_{\mathcal X_{T,Q}},
        \qquad
        b_\tau(N,Q)\ge0.
\end{equation}
The power-law Volterra criterion of Definition~\ref{def:contraction-resolution} and Proposition~\ref{prop:diagonal} implies \eqref{eq:env-retained-bound} and \eqref{eq:env-X-localized} in the corresponding power-law cases.  The intersection \(X_T^\sigma\cap\mathcal X_T\) is equipped with the sum norm.

\begin{theorem}[Frequency-envelope variant]\label{thm:envelope}
Fix \(0<T\le T_0\) and a finite family \(\mathcal M\) of mixed labels.
Assume the Fourier-diagonal covariance model \eqref{eq:conv-cov} together with the envelope coefficient bounds \eqref{eq:env-K}--\eqref{eq:env-A-increment}.  Let \((\beta_\tau,\zeta_\tau,\sigma)\) be such that, for every \(\tau\in\mathcal M\), the two quantities \(\mathfrak S_{\tau,H}^{\eps}(\beta_\tau,\sigma)\) and \(\mathfrak S_{\tau,B}^{\eps}(\zeta_\tau,\sigma)\) in \eqref{eq:env-SH}--\eqref{eq:env-SB} are finite, with \(\eps>0\) leaving room for the time-lift absorption \eqref{eq:time-lift-absorb}.  Then, for every finite \(p\ge2\), the centered cutoff operators satisfy
\begin{equation}\label{eq:env-centered-conv}
        B_\Lambda^\tau\longrightarrow B^\tau
        \quad\hbox{in}\quad
        L^p\bigl(\Omega;\mathcal L(X_T^\sigma,Y_T^{\beta_\tau,\zeta_\tau})\bigr)
\end{equation}
and almost surely in the corresponding operator norm.  If, in addition, a deterministic diagonal resolution has been chosen so that the retained branches satisfy \eqref{eq:env-retained-bound}, the localized source-tail bound \eqref{eq:env-X-localized}, and fixed-block convergence in \(Y_T^{\beta_\tau,\zeta_\tau}\), then
\begin{equation}\label{eq:env-assembled-conv}
        T_\Lambda^{\tau,\mathfrak P}=T_\Lambda^\tau-\Ctr^\tau_\Lambda
        \longrightarrow \Ret^\tau+B^\tau
        \quad\hbox{in}\quad
        \mathcal L(X_T^\sigma\cap\mathcal X_T,Y_T^{\beta_\tau,\zeta_\tau})
\end{equation}
almost surely and in \(L^p(\Omega)\).  Fully Wick-subtracted branches are obtained by setting \(\Ret^\tau_\Lambda=0\).
\end{theorem}

\subsection{Proof of the frequency-envelope variant}\label{subsec:proof-envelope}

We prove Theorem~\ref{thm:envelope} by repeating the preceding argument with the power weights replaced by the prescribed envelopes.

The Wick split is unchanged.  Fourier-diagonal covariance gives
\(\ell+r=0\) in the scalar contraction, hence \(n=q\) for the deterministic
diagonal branch, while the centered branch keeps the incidence
\(n=q+\ell+r\).  Lemma~\ref{lem:color-split} and
\eqref{eq:finite-ren-equals-B} give the same finite decomposition.

The local normalization of Lemma~\ref{lem:local-normal-form} is performed
after factoring out the envelope sizes.  For \(|\xi|\sim N\), write
\[
        K_i(t,\xi)=K_i(N)\bar K_i(t,\xi),
        \qquad
        h_{a,\xi,t}=A_a(N)\bar h_{a,\xi,t},
\]
where the normalized quantities are uniformly bounded in the sense of \eqref{eq:env-K} and \eqref{eq:env-A}.  Hence the centered coefficient tensor satisfies
\begin{equation}\label{eq:env-H-bound}
        |H_{\nu;a,b,c,e}(t,s)|
        \lesssim G_\tau(N)
\end{equation}
with \(G_\tau(N)=K_i(N)A_j(N)A_k(N)\).  The whitening remains local on each Fourier involution class and therefore does not change the spatial incidence.

Proposition~\ref{prop:tensor} and Lemma~\ref{lem:incidence-flattenings}
give the fixed-time matrix estimate
\begin{equation}\label{eq:env-fixed-matrix}
\Bigl\|\,\|B_{\Lambda,N,Q,M}^\tau(t,s)\|_{\ell_q^2\to\ell_n^2}\Bigr\|_{L^p(\Omega)}
\lesssim_{p,\eps}
        N^{d/2+\eps}G_\tau(N)
        \bigl(M^{d/2+\eps}+Q^{d/2+\eps}\bigr).
\end{equation}
The logarithmic factor in Proposition~\ref{prop:tensor} is again absorbed into the displayed \(\eps\)-loss.

For the time lift, expand a time difference by the finite telescope used
above.  Each summand contains either a Duhamel increment bounded by
\eqref{eq:env-K-increment} or a Gaussian-leg increment bounded by
\eqref{eq:env-A-increment}.  Thus the fixed-time majorant in
\eqref{eq:env-fixed-matrix} is multiplied by
\(\Theta_\tau(N)^\theta |t-t'|^\theta\) or by
\(\Theta_\tau(N)^\theta |s-s'|^\theta\).  Assumption
\eqref{eq:time-lift-absorb} absorbs this factor into an additional
\(N^\eps\) in Lemma~\ref{lem:time-lift}.  After decreasing the displayed
\(\eps\) in the summability assumptions, we obtain
\begin{equation}\label{eq:env-time-matrix}
\Bigl\|\,\|B_{\Lambda,N,Q,M}^\tau\|_{C_{t,s}\cL(\ell_q^2,\ell_n^2)}\Bigr\|_{L^p(\Omega)}
\lesssim_{p,T,\eps}
        N^{d/2+\eps}G_\tau(N)
        \bigl(M^{d/2+\eps}+Q^{d/2+\eps}\bigr).
\end{equation}

Inserting a deterministic input and using
\(\|P_Qw\|_{L_T^\infty L^2}\le Q^{-\sigma}\|w\|_{X_T^\sigma}\),
\eqref{eq:env-time-matrix} gives
\begin{equation}\label{eq:env-inserted-block}
        \|P_MB_{N,Q,M}^\tau P_Qw\|_{L^p(\Omega;C_TL^2)}
        \lesssim
        N^{d/2+\eps}G_\tau(N)
        \bigl(M^{d/2+\eps}+Q^{d/2+\eps}\bigr)
        Q^{-\sigma}\|w\|_{X_T^\sigma} .
\end{equation}
Multiplication by the Sobolev weight \(M^{\beta_\tau}\) and summation over \((N,Q,M)\) is precisely \(\mathfrak S_{\tau,H}^{\eps}(\beta_\tau,\sigma)\).  For the Besov component, define the shell random variable
\[
        X_M= M^{\zeta_\tau}
        \left\|P_M\sum_{N\gtrsim M}\sum_{Q\le c_{\rm ap}N}
        B_{N,Q,M}^\tau P_Q\right\|_{\cL(X_T^\sigma,L_T^1L^2)} .
\]
By Minkowski's inequality and \eqref{eq:env-inserted-block}, \(\|X_M\|_{L^p(\Omega)}\) is bounded by \(\mathfrak b_{\tau,M}^{\eps}(\zeta_\tau,\sigma)\).  Therefore
\[
        \left\|\sup_M X_M\right\|_{L^p(\Omega)}
        \le \sum_M \|X_M\|_{L^p(\Omega)}
        \le \mathfrak S_{\tau,B}^{\eps}(\zeta_\tau,\sigma).
\]
Together with the Sobolev summation this gives the uniform centered bound in
\[
        L^p\bigl(\Omega;\cL(X_T^\sigma,Y_T^{\beta_\tau,\zeta_\tau})\bigr).
\]

For convergence, fix a dyadic triple \((N,Q,M)\).  The Galerkin block is
eventually independent of \(\Lambda\).  The \(\ell^1\) majorant
\(\mathfrak S_{\tau,H}^{\eps}\) controls the Sobolev tail, while
\(\mathfrak S_{\tau,B}^{\eps}\) controls the Besov output tail.
Lemma~\ref{lem:dyadic-tail-assembly} therefore gives the \(L^p\)-convergence
in \eqref{eq:env-centered-conv}.

For the almost-sure assertion, apply the \(p=2\) block estimate to the finite
full and boundary-cutoff family \(\mathscr C_{N,Q,M}\) used in the proof of
Theorem~\ref{thm:main}.  Multiply its supremum by
\(M^{\beta_\tau}Q^{-\sigma}\) for the Sobolev component and by
\(M^{\zeta_\tau}Q^{-\sigma}\) for the Besov component.  The expectations of
the resulting random majorants are summable, respectively, by
\(\mathfrak S_{\tau,H}^{\eps}<\infty\) and
\(\mathfrak S_{\tau,B}^{\eps}<\infty\).  Fubini gives pathwise summability,
and Lemma~\ref{lem:dyadic-tail-assembly} then yields almost-sure
operator-norm convergence.

If a deterministic diagonal resolution is given, \eqref{eq:env-retained-bound}, \eqref{eq:env-X-localized} and fixed-block convergence imply convergence of the retained branch by the same finite-block plus summable-tail argument.  Indeed, the tail outside a finite dyadic set is bounded by
\[
        C_{\mathfrak P}\sum_{(N,Q)\notin F} b_\tau(N,Q)\|P_Qw\|_{\mathcal X_{T,Q}},
\]
and this tends to zero uniformly for \(\|w\|_{\mathcal X_T}\le1\) by \eqref{eq:env-X-localized}.  Adding this deterministic limit to the centered limit and using \eqref{eq:finite-ren-equals-B} proves \eqref{eq:env-assembled-conv}.  The case \(\Ret^\tau_\Lambda=0\) is the fully Wick-subtracted branch.  This proves Theorem~\ref{thm:envelope}.

\phantomsection
\pdfbookmark[1]{References}{references}

\end{document}